\documentclass[reqno]{amsart}

\usepackage{algorithmic}
\usepackage{tabu}
\usepackage{amssymb}
\usepackage{mathtools}
\usepackage{a4wide,amsmath}
\usepackage{mathrsfs}
\usepackage{amsthm}
\numberwithin{equation}{section}
\numberwithin{figure}{section}
\numberwithin{table}{section}
\usepackage{bbm}
\usepackage{subfig}
\usepackage{enumerate}
\usepackage[section]{placeins}
\usepackage{graphicx}		  
\usepackage{ifpdf}
\ifpdf
\DeclareGraphicsExtensions{.pdf,.eps,.jpg,.png}	
\usepackage[suffix=]{epstopdf}
\fi
\usepackage{xcolor}
\usepackage[utf8]{inputenc}
\usepackage{hyperref}
\hypersetup{hidelinks}

\long\def\MSC#1\EndMSC{\def\arg{#1}\ifx\arg\empty\relax\else
{\narrower\noindent%
{2010 Mathematics Subject Classification}: #1\\} \fi}
\long\def\PACS#1\EndPACS{\def\arg{#1}\ifx\arg\empty\relax\else
{\narrower\noindent%
{PACS numbers}: #1}\fi}
\long\def\KEY#1\EndKEY{\def\arg{#1}\ifx\arg\empty\relax\else
{\narrower\noindent%
Keywords: #1\\}\fi}

\theoremstyle{plain}
\newtheorem{theorem}{Theorem}[section]
\newtheorem{lemma}[theorem]{Lemma}
\newtheorem{proposition}[theorem]{Proposition}
\newtheorem{corollary}[theorem]{Corollary}
\theoremstyle{definition}
\newtheorem{definition}[theorem]{Definition}
\newtheorem{example}[theorem]{Example}

\theoremstyle{remark}
\newtheorem{remark}[theorem]{Remark}

\newcommand{\R}{\mathbb{R}}
\newcommand{\Z}{\mathbb{Z}}
\newcommand{\N}{\mathbb{N}}
\newcommand{\C}{\mathbb{C}}
\newcommand{\norm}[1]{\lVert#1\rVert}
\newcommand{\abs}[1]{\lvert#1\rvert} 
\newcommand{\inner}[1]{\langle#1\rangle}

\newcommand{\redel}{\mathop{\textup{Re}}}

\newcommand{\mspan}{\mathop{\textup{span}}}

\newcommand{\ident}{\mathop{\textup{id}}}

\newcommand{\suppm}{\mathop{\textup{supp}}}

\newcommand{\I}{\mathrm{i}}    
\newcommand{\e}{\mathrm{e}}    
\newcommand{\di}{\mathrm{d}}   

\begin{document}

\title[Mimicking continuum measurements by electrode data]{Mimicking relative continuum measurements by electrode data in two-dimensional electrical impedance tomography}

\author[H.~Garde]{Henrik Garde}
\address[H.~Garde]{Department of Mathematics, Aarhus University, Ny Munkegade 118, 8000 Aarhus C, Denmark.}
\email{garde@math.au.dk}

\author[N.~Hyv\"onen]{Nuutti Hyv\"onen}
\address[N.~Hyv\"onen]{Department of Mathematics and Systems Analysis, Aalto University, P.O. Box~11100, 00076 Helsinki, Finland.}
\email{nuutti.hyvonen@aalto.fi}

\begin{abstract}
  This paper introduces a constructive method for approximating relative continuum measurements in two-dimensional electrical impedance tomography based on data originating from either the point electrode model or the complete electrode model. The upper bounds for the corresponding approximation errors explicitly depend on the number (and size) of the employed electrodes as well as on the regularity of the continuum current that is mimicked. In particular, if the input current and the object boundary are infinitely smooth, the discrepancy associated with the point electrode model converges to zero faster than any negative power of the number of electrodes. The results are first proven for the unit disk via trigonometric interpolation and quadrature rules, and they are subsequently extended to more general domains with the help of conformal mappings.
\end{abstract}
	
\maketitle

\KEY
electrical impedance tomography, 
electrode placement,
interpolation,
quadrature rule
\EndKEY

\MSC
35R30, 35A35, 42A15, 42A10
\EndMSC

\section{Introduction}

In \emph{electrical impedance tomography} (EIT) the goal is to reconstruct an unknown conductivity distribution inside a physical body. This is based on noninvasive measurements of electric current and voltage at electrodes placed on the boundary of the object.  According to the idealized \emph{continuum model} (CM) of EIT, such measurements correspond to knowing the infinite-dimensional \emph{Neumann-to-Dirichlet} (ND) boundary map for the conductivity equation that models the behavior of the electric potential inside the object of interest. Alternatively, if one is able to perform the measurements of EIT for the same object with two different conductivities, the CM assumes the available data is the difference of the respective ND maps, i.e.~the so-called {\em relative measurements}. This is the case with time-difference imaging or by numerically simulating the reference measurements for, say, the unit conductivity. Quite often the relative measurements are supposed to originate from two conductivities that coincide in some interior neighborhood of the object boundary, which is also the setting considered in this work.

In the framework of the CM, an isotropic conductivity, i.e.~the strictly positive scalar-valued coefficient in the conductivity equation, is uniquely identified by the corresponding boundary measurements under dimension-dependent regularity assumptions. In two dimensions, it is enough to assume the conductivity is essentially bounded \cite{Astala2006a}, whereas the most general uniqueness results in higher dimension require Lipschitz continuity \cite{CaroRogers2016}. There also exist several reconstruction methods that are based on the (slightly unrealistic) CM such as the $\overline{\partial}$-method \cite{Knudsen2007,Knudsen2009,Nachman1996}, the monotonicity method \cite{Garde_2019b,Harrach10,Harrach13,Candiani2019}, the factorization method \cite{Kirsch2008,Bruhl2000,Bruhl2001,Hanke2015,Gebauer2007}, and the enclosure method \cite{Ikehata1998a,Ikehata1999a,Ikehata2000c}. It is worth noting that all these reconstruction techniques actually assume the availability of {\em relative} CM measurements.  

In all real-world settings for EIT, it is only possible to perform measurements with a finite number of finite-sized electrodes, making it impossible to exactly record continuum data in practice. In consequence, more accurate electrode models should arguably be used to treat such realistic measurement setups. The \emph{complete electrode model} (CEM) includes the actual shapes and positions of the employed electrodes as well as the contact impedances at the electrode-object interfaces in the forward model~\cite{Cheng1989,Somersalo1992}. The limit of the CEM when the contact impedances tend to zero is called the {\em shunt model} \cite{Cheng1989}; a bit counterintuitively, the shunt model has been shown to exhibit worse numerical behavior than the CEM \cite{Darde2016}. For small or point-like electrodes and relative measurements, the injected currents can be relatively accurately modeled as Dirac delta distributions according to the \emph{point electrode model} (PEM) \cite{Hanke2011b}.

Although the aforementioned electrode models accurately predict real-world measurements under appropriate assumptions on the measurement configuration, very little is actually known about the unique identifiability of conductivities in their frameworks: among the only such results, the unique identifiability has been proven in \cite{Harrach_2019} for the CEM assuming the conductivity belongs to a suitable a priori known finite-dimensional set of piecewise analytic functions. In fact, it is intuitively acceptable to assume there is no uniqueness for any combination of an {\em infinite-dimensional} family of conductivities and an electrode model producing {\em finite-dimensional} data. For the PEM it is even possible to stably construct conductivity perturbations that really are invisible to measurements by a given finite set of electrodes \cite{Chesnel_2015}. However, since it is also possible to approximate (absolute) CM measurements by those of the CEM when the number of electrodes tends to infinity and the electrodes cover the object boundary in a controlled manner, one can argue that the unique identifiability results for the CM transfer to the framework of the CEM in the sense of limits. To be slightly more precise, it has been shown in~\cite{Hyvonen09} that when the number of electrodes is increasing (and their size appropriately decreasing) the ND map corresponding to a smooth enough conductivity can be approximated in the space of bounded linear operators on square integrable functions, so that the associated discrepancy tends to zero essentially as the maximal distance between adjacent electrodes.

As mentioned above, many of the reconstruction methods designed for the CM assume relative measurements as their input, and so the ability to approximate relative continuum data by electrode models is important for the practical implementation of sophisticated mathematical reconstruction algorithms. In fact, electrode measurement variants have previously been introduced for, e.g., the factorization method, the monotonicity method, and the $\overline{\partial}$-method (see,~e.g.,~\cite{Lechleiter2008a,GardeStaboulis_2016,Garde_2019,Harrach15,Isaacson2004}), but the accuracy of the associated techniques for approximating relative continuum data based on the available electrode measurements has not always been carefully analyzed. Indeed, the best result in this direction the authors are aware of is the one in~\cite{Hyvonen09}, which does not, in particular, exploit the extra structure carried by relative measurements if the considered two conductivities coincide in an interior neighborhood of the object boundary. See also \cite{Calvetti2019} for a recent Bayesian approach for moving between electrode and continuum measurements in EIT. 

In this paper, we tackle mimicking relative continuum data by electrode measurements in two-dimensional EIT as a problem of (hardware) algorithm design: our aim is to choose (optimal) positions for the employed electrodes based on the shape of the imaged domain and the net electrode currents as functions of the continuum current pattern one would like to drive through the object boundary. We assume the examined bounded simply connected two-dimensional domain has a smooth enough boundary and the target continuum current patterns exhibit $L^2$-based Sobolev regularity of order $s>\tfrac{1}{2}$. Our algorithm leads to estimates of order $\mathcal{O}(M^{-s})$ in the number of electrodes $M$ for the discrepancy between the relative measurements of the CM and suitably postprocessed PEM or CEM data (provided the width of the electrodes decay appropriately in $M$ for the CEM). This result is first proven for the PEM in the unit disk with equiangular electrodes, and subsequently it is extended to more general domains with the help of conformal mappings. Finally, the required estimates for the CEM are obtained by resorting to the material in \cite{Hanke2011b}, where an approximative link between the CEM and the PEM is considered. In addition to the Riemann mapping theorem, the main ingredients for obtaining our estimates are sufficiently accurate interpolation and quadrature rules on the boundary of the unit disk for the relevant Sobolev spaces; in fact, the order of the obtained estimates is directly dictated by the accuracy of these rules.


This paper is organized as follows. In Section~\ref{sec:forward}, we introduce the CM, PEM, and CEM as well as their respective relative measurement maps. Section~\ref{sec:Sobolev} briefly recalls the connection between Sobolev spaces and Fourier series, while Section~\ref{sec:approxuniform} reviews trigonometric interpolation and introduces estimates related to pointwise current injection in the PEM. Section~\ref{sec:unit_disk} provides the desired estimates for the PEM in the unit disk as Theorem~\ref{thm:est1}. Finally, Section~\ref{sec:general_domain} introduces our algorithm with general domains for both the PEM and the CEM, with Theorem~\ref{thm:main} and Corollary~\ref{cor:main} serving as our main results. The paper is concluded by two numerical examples in Section~\ref{sec:numer}, verifying some of the proven convergence rates in simple geometries. Finally, an appendix clarifies how the constants appearing in the estimates of \cite{Hanke2011b} depend on the number of electrodes.

\subsection{Some notational remarks} 

Let $\mathscr{L}(X,Y)$ denote the space of bounded linear operators between Banach spaces $X$ and $Y$, with $\mathscr{L}(X) := \mathscr{L}(X,X)$.

For the sake of brevity, we use the following notation for finite summations with $M\in\mathbb{N}$:
\begin{equation*}
\sum_{\abs{n}\leq M} a_n := \sum_{n=-M}^M a_n, \qquad \sum_{0<\abs{n}\leq M} a_n := \sum_{n=-M}^{-1} a_n + \sum_{n=1}^{M} a_n.
\end{equation*}
An analogous notation is also used for sets that are indexed by integers.  For $M\in\mathbb{N}$, we systematically index the components of vectors in $\C^{2M+1}$ from $-M$ to $M$, i.e.~$a = [a_{-M}, \dots, a_{M}]^{\rm T} \in \C^{2M+1}$, and denote the mean free subspace of $\C^{2M+1}$ by $\C_\diamond^{2M+1}$.

We often use generic positive constants that may change from one estimate to the next. As an example, writing~$C(a,b,c)$ indicates that such a constant only depends on the parameters $a$, $b$, and $c$.

\section{On forward models of EIT} \label{sec:forward}

In this section we review three forward models for (two-dimensional) EIT. All of them correspond to the same underlying elliptic partial differential equation
\begin{equation} \label{eq:PDE}
	\nabla \cdot (\sigma \nabla u) = 0 \qquad \text{in } \Omega,
\end{equation}
where $\Omega \subset \C$ is a bounded simply connected domain with a $C^\infty$-boundary. The coefficient  $\sigma \in L^\infty(\Omega)$ is a complex-valued isotropic conductivity satisfying
\begin{equation} \label{eq:sigma_condition1}
	\suppm(\sigma - 1) \subset  \subset \Omega \qquad \text{and} \qquad \varsigma_{-} \leq \redel(\sigma) \leq \varsigma_{+}
\end{equation}
almost everywhere in $\Omega$ for some $\varsigma_{-},\varsigma_{+} \in \R_+$.

\begin{remark}
	The forward problems of CM and CEM would still be well defined if $\partial \Omega$ were Lipschitz and only the second condition of \eqref{eq:sigma_condition1} were satisfied (assuming the boundary current densities for the CM are regular enough). On the other hand, the PEM requires the first condition of \eqref{eq:sigma_condition1} and also some regularity from $\partial \Omega$, but piecewise $C^{1, \alpha}$-smooth boundary would actually be enough \cite{Seiskari2014}. Moreover, all three models could be introduced in three dimensions and for anisotropic conductivities as well. The reason for assuming two-dimensionality and extra regularity from $\sigma$ and $\partial \Omega$ is that they are required by our main results in their optimal form. On the other hand, anisotropic conductivities are excluded solely for the sake of notational convenience.
\end{remark}

\subsection{Continuum model} \label{sec:CM}

The CM assumes that one can drive any mean free normal current density through $\partial \Omega$ and measure the resulting boundary potential everywhere. In mathematical terms, \eqref{eq:PDE} is accompanied by the Neumann boundary condition
\begin{equation} \label{eq:Neumann}
	\frac{\partial u}{\partial \nu} = f  \qquad \text{on } \partial \Omega,  
\end{equation}
where $\nu$ is the exterior unit normal of $\partial \Omega$ and $f \in H^s_\diamond(\partial \Omega)$, with $H^s_\diamond(\partial \Omega)$ denoting the mean free subspace of the Sobolev space $H^s(\partial \Omega)$. To be more precise,
\begin{equation*}
	H^s_\diamond(\partial \Omega) := \{ v \in H^s_\diamond(\partial \Omega) \mid \inner{1,v}_{\partial \Omega}=0 \}, \qquad s \in \R,
\end{equation*}
where $\inner{\cdot, \cdot}_{\partial \Omega}: H^{-s}(\partial \Omega) \times H^s(\partial \Omega) \to \C$ denotes the sesquilinear dual evaluation between the associated Sobolev spaces.

According to the standard theory on elliptic boundary value problems \cite{Necas2012}, the combination of \eqref{eq:PDE} and \eqref{eq:Neumann} has a unique solution $u \in H^{{\min\{1, s+ 3/2\}}}(\Omega)/\C$. Moreover, as $u$ satisfies the Laplace equation in an interior neighborhood of $\partial \Omega$, its Dirichlet trace is well-defined in $H^{s + 1}(\partial \Omega)/\C$~\cite{Lions1972}. By choosing the ground level of potential appropriately, one can thus introduce the ND map
\begin{equation*} 
	\Lambda(\sigma): \begin{cases}
		f \mapsto u|_{\partial \Omega}, \\[1mm]
		H^s_\diamond(\partial \Omega) \to  H^{s+1}_\diamond(\partial \Omega),
	\end{cases}
\end{equation*}
whose boundedness for any $s$ is considered,~e.g.,~in \cite[Theorem~A.3]{Hanke2011}.

The relative ND map
\begin{equation} \label{eq:T}
	\Upsilon(\sigma) := \Lambda(\sigma) - \Lambda(1)
\end{equation}
exhibits considerably more regularity than $\Lambda(\sigma)$ itself. More precisely,
\begin{equation} \label{eq:Tbound}
	\norm{\Upsilon(\sigma)}_{\mathscr{L}(H_\diamond^{t}(\partial \Omega), H_\diamond^{s}(\partial \Omega))} \leq C(s,t, \sigma, \Omega)
\end{equation}
for all $t,s \in \R$. This well known result is presented in \cite[Theorem~A.3]{Hanke2011} for the unit disk, but the corresponding proof applies as such to any smooth and bounded domain $\Omega$ since the essential ingredient is the first condition in \eqref{eq:sigma_condition1}, not the shape of the domain.

According to the CM, relative EIT measurements for a conductivity satisfying \eqref{eq:sigma_condition1} produce the boundary operator $\Upsilon(\sigma)$ as the data. However, it is obvious that $\Upsilon(\sigma)$ cannot be precisely retrieved based on practical EIT measurement that employ a finite set of contact electrodes. This observation motivates the introduction of more realistic electrode models for EIT.

\begin{remark}
  The estimate \eqref{eq:Tbound} remains valid if the unit background conductivity in \eqref{eq:sigma_condition1} and \eqref{eq:T} is replaced by a smooth conductivity $\sigma_0 \in C^\infty(\overline{\Omega})$, whose real part is positive and bounded away from zero, and the left-hand side of \eqref{eq:Neumann} is weighted by $\sigma_0$ to properly model the corresponding boundary current density. As the smoothening property \eqref{eq:Tbound} is a main ingredient of our analysis, it is straightforward to check that all results presented below still hold within such a slightly more general framework (if the unit conductivity is replaced by $\sigma_0$ at every occurrence, including all Neumann boundary conditions).
  \end{remark}

\subsection{Point electrode model} \label{sec:PEM}

In our framework, the PEM assumes that $2M+1$ point-like electrodes are attached to $\partial \Omega$ at the distinct locations $y_{-M}, \dots, y_M \in \partial \Omega$ for some $M\in\mathbb{N}$. The net currents $I_{-M}, \dots, I_M \in \C$, with a zero-mean condition imposed, are driven through the electrodes and the relative potentials are measured at these same positions. In other words, the employed current patterns are of the form
\begin{equation} \label{eq:f_I}
	f_I := \sum_{\abs{m}\leq M} I_m  \delta_{y_m},\qquad I \in \C_\diamond^{2M+1},
\end{equation}
where $\delta_{y_m} \in H^{s}(\partial \Omega)$, $s<-\frac{1}{2}$, is the Dirac delta distribution supported at $y_m$ on the one-dimensional boundary $\partial \Omega$. Observe that $f_I \in H_\diamond^{s}(\partial \Omega)$, $s<-\frac{1}{2}$, for any $I\in\C_\diamond^{2 M+1}$.

Recalling the relative ND map from \eqref{eq:T}, the measurements of PEM can be modeled by the pointwise current-to-voltage map
\begin{equation} \label{eq:Upsilon}
	\Upsilon_{\rm PEM}(\sigma): \C_\diamond^{2M+1} \ni I \mapsto
	\begin{bmatrix}
		\Upsilon(\sigma) f_I (y_{-M}) \\
		\vdots \\
		\Upsilon(\sigma) f_I (y_M)
	\end{bmatrix}
	+ c \mathbf{1} \in \C_\diamond^{2M+1},
\end{equation}
where $\mathbf{1} = [1, \dots, 1]^{\rm T} \in \C^{2M+1}$ and the ground level of potential $c \in \C$ is chosen so that the relative potentials at the electrodes have zero mean. Due to \eqref{eq:Tbound} and the Sobolev embedding theorem, the definition \eqref{eq:Upsilon} is unambiguous.

According to the PEM, relative EIT measurements for a conductivity satisfying \eqref{eq:sigma_condition1} produce (a noisy version of) the finite-dimensional linear map $\Upsilon_{\rm PEM}(\sigma)$ as the data. Although the point-like electrodes of the PEM cannot completely accurately model the finite-sized ones used in practical EIT measurements, it has been shown that the discrepancy between relative measurements modeled by the CEM and the PEM behave asymptotically as $\mathcal{O}(d^2)$ in the maximal diameter $d>0$ of the electrodes \cite{Hanke2011b}.

\subsection{Complete electrode model} \label{sec:CEM}

The CEM is arguably the most accurate model for EIT \cite{Cheng1989,Somersalo1992}. In our two-dimensional setting, $2M+1$ mutually disjoint electrodes $E_{-M}, \dots, E_M \subset \partial \Omega$ are attached to the object boundary. They are identified with the nonempty open and connected  subsets of $\partial \Omega$ that they cover, and their midpoints with respect to the arclength of $\partial \Omega$ are $y_{-M}, \dots, y_M \in \partial \Omega$, respectively. The contact impedances at the electrode-object interfaces are modeled by the complex numbers $z_{-M}, \dots, z_M \in \C$ with positive real parts.

In the forward problem of the CEM, the conductivity equation \eqref{eq:PDE} is combined with the boundary conditions
\begin{alignat}{2}
	\frac{\partial u}{\partial \nu} &= 0 &\qquad &\text{on} \ \partial \Omega\setminus \smashoperator{\bigcup_{\abs{m} \leq M}} \overline{E_m}, \notag\\[1mm] 
u+z_m \dfrac{\partial u}{\partial \nu}&= U_m & &\text{on} \ E_m, \quad \abs{m}\leq M, \label{cemeqs}\\[3mm] 
\int_{E_m} \dfrac{\partial u}{\partial \nu} \, {\rm d} s &= I_m, & & \abs{m}\leq M, \notag
\end{alignat}
where $I \in \C_\diamond^{2M+1}$ models the net currents through the electrodes, $u \in H^1(\Omega)$ is the electric potential within $\Omega$, and $U \in \C_\diamond^{2M+1}$ carries the constant potentials at the electrodes. The combination of \eqref{eq:PDE} and \eqref{cemeqs} uniquely defines the interior-electrode potential pair $(u, U) \in H^1(\Omega) \oplus \C_\diamond^{2M+1}$ \cite{Somersalo1992}. Take note that forcing the electrode potential vector to belong to $\C_\diamond^{2M+1}$ corresponds to a particular choice for the ground level of potential.

The absolute measurements of the CEM are modeled by the electrode current-to-voltage map
\begin{equation*}
	R(\sigma): \C_\diamond^{2M+1} \ni I \mapsto U \in \C_\diamond^{2M+1},
\end{equation*}
and the corresponding relative measurement map is
\begin{equation} \label{eq:Upsilon_CEM}
\Upsilon_{\rm CEM}(\sigma) := R(\sigma) - R(1).
\end{equation}
According to the CEM, relative EIT measurements for a conductivity satisfying \eqref{eq:sigma_condition1} produce (a noisy version of) the finite-dimensional linear map $\Upsilon_{\rm CEM}(\sigma)$ as the data.

\section{On periodic distributions, periodic Sobolev spaces, and Fourier series} \label{sec:Sobolev}

This section reviews some basic facts on periodic distributions and Sobolev spaces based on \cite[Sections 5.2 and 5.3]{Saranen2002}. We present the theory in terms of $2\pi$-periodic functions and distributions rather than 1-periodic ones as in \cite{Saranen2002}, mainly because we will initially focus on the unit disk when considering EIT.

We denote by $\mathcal{D}(\mathbb{R}) := C_\textup{c}^\infty(\mathbb{R})$ the space of smooth functions with compact support and by $\mathcal{D}'(\mathbb{R})$ its dual space, i.e.\ the space of distributions on $\mathbb{R}$. The dual pairing between these spaces is denoted as $\inner{\cdot,\cdot}_\mathbb{R}$.
\begin{definition}
	A distribution $g\in \mathcal{D}'(\mathbb{R})$ is called $2\pi$-periodic if 
	\begin{equation*}
		\inner{g,\tau_k\phi}_\mathbb{R} = \inner{g,\phi}_\mathbb{R}, \qquad \forall \phi\in \mathcal{D}(\mathbb{R}), \enskip k\in\mathbb{Z},
	\end{equation*}
	where $\tau_k\phi(\theta) := \phi(\theta+2\pi k)$ for $\theta\in\mathbb{R}$. Moreover, the spaces of $2\pi$-periodic smooth functions and distributions are defined as
	\begin{align*}
		\mathcal{D}_\textup{per} &:= \{ \phi \in C^\infty(\mathbb{R}) \mid \phi \text{ is $2\pi$-periodic} \}, \\[1mm]
		\mathcal{D}'_\textup{per} &:= \{ g\in \mathcal{D}'(\mathbb{R}) \mid g \text{ is $2\pi$-periodic} \},
	\end{align*}
	respectively.
\end{definition}
As hinted by the notation, $\mathcal{D}'_\textup{per}$ can be identified with the dual space of $\mathcal{D}_\textup{per}$: Following \cite[Section~5.2]{Saranen2002}, we may introduce $\psi\in \mathcal{D}(\mathbb{R})$ such that its $2\pi$-translates form a partition of unity, i.e.\
\begin{equation}
	\sum_{k\in\mathbb{Z}} \psi(\theta + 2\pi k) \equiv 1,\qquad \theta\in\mathbb{R}. \label{eq:partitionofunity}
\end{equation}
The dual pairing between $\mathcal{D}'_\textup{per}$ and $\mathcal{D}_\textup{per}$ is then defined by
\begin{equation}
	\inner{g,\phi} := \inner{g,\psi\phi}_\mathbb{R}, \qquad g\in \mathcal{D}'_\textup{per}, \enskip \phi \in \mathcal{D}_\textup{per}. \label{eq:dualperiodicdist}
\end{equation}
A short computation reveals that the definition in \eqref{eq:dualperiodicdist} is independent of the choice of $\psi$ with the property \eqref{eq:partitionofunity}. We refer to \cite[Section~5.2]{Saranen2002} for a more careful analysis of the duality between $\mathcal{D}'_\textup{per}$ and $\mathcal{D}_\textup{per}$.

The definition \eqref{eq:dualperiodicdist} enables introducing Fourier coefficients for periodic distributions. To streamline the notation, we denote the trigonometric monomials by $f_n := \e^{\I n (\cdot)}\in \mathcal{D}_\textup{per}$, $n\in\mathbb{Z}$.
\begin{definition} \label{defi:fouriercoeff}
	The Fourier coefficients of $g\in \mathcal{D}'_\textup{per}$ are defined as
	\begin{equation}
		\hat{g}(n) := \frac{1}{2\pi}\inner{g,f_{-n}}, \label{eq:fouriercoeff} \qquad n  \in \Z.
	\end{equation}
\end{definition}

\begin{remark}[Related to Definition~\ref{defi:fouriercoeff}] {}\

\begin{enumerate}[(i)]
	\item By identifying $L^1_\textup{loc}(\R)$ as a subspace of $\mathcal{D}'(\R)$ in the usual manner, it is easy to see that \eqref{eq:fouriercoeff} provides a generalization for the standard Fourier coefficients defined for $2\pi$-periodic $L^1_\textup{loc}(\R)$-functions.
	\item We identify the standard Dirac delta distribution $\delta_{\theta_0}$, $\theta_0\in (-\pi,\pi]$, with its $2\pi$-periodic counterpart
	\begin{equation*}
	\sum_{k\in \mathbb{Z}} \delta_{\theta_0+2\pi k}.
	\end{equation*}
	It follows immediately from \eqref{eq:fouriercoeff} that
	\begin{equation}\label{eq:F_delta}
	\hat{\delta}_{\theta_0}(n) = \frac{1}{2\pi}\e^{-\I n\theta_0}, \qquad n\in\mathbb{Z}.
	\end{equation}
\end{enumerate}
\end{remark}

Let us then define the periodic Sobolev spaces.
\begin{definition}
	The $2\pi$-periodic Sobolev space with a smoothness index $s \in \R$ is defined as
	\begin{equation*}
		H^s := \{ g \in \mathcal{D}'_\textup{per} \, \mid \, \norm{g}_s < \infty \},
	\end{equation*}
    where 
	\begin{equation} \label{eq:s_norm}
		\norm{g}_s := \biggl( 2 \pi \sum_{n\in\mathbb{Z}} \underline{n}^{2s}\abs{\hat{g}(n)}^2\biggr)^{1/2}
	\end{equation}
	and $\underline{n} := \max\{1,\abs{n}\}$.
	The mean free subspace of $H^s$ is
	\begin{equation*}
		H^s_\diamond := \{ g\in H^s \, \mid \, \hat{g}(0) = 0 \}.
	\end{equation*}
\end{definition}
It is obvious that $H^s$ (as well as its closed subspace $H_\diamond^s$) is  a Hilbert space when equipped with the inner product
\begin{equation}
	\inner{f,g}_s := 2 \pi \sum_{n\in\mathbb{Z}}\underline{n}^{2s}\hat{f}(n)\overline{\hat{g}(n)}, \qquad f,g\in H^s. \label{eq:innerprod}
\end{equation}
We write $L^2 := H^0$ and $L^2_{\diamond} := H^0_\diamond$, which is motivated by Parseval's theorem guaranteeing that $\inner{f,g}_0 = \inner{f,g}_{L^2}$ for all $f,g \in L^2(-\pi,\pi)$, that is, $H^0 \simeq L^2(-\pi,\pi)$.

The bracket $\inner{\cdot,\cdot}_{-s,s} : H^{-s}\times H^s\to \mathbb{C}$ denotes the sesquilinear dual pairing, acting as an extension of the inner product $\inner{\cdot,\cdot}_{L^2}$ on $L^2(-\pi,\pi)$. By mimicking the construction in \cite[Section~1.1]{Garde_2019c}, one can use this same pairing for realizing the duality between $H_\diamond^{-s}$ and $H_\diamond^s$ as well. 

\begin{remark} \label{remark:lions}
	We may identify $H^s$ with the standard Sobolev space $H^s(\partial D)$ on the boundary of the unit disk $D$. Indeed, via the identification $\theta \leftrightarrow \e^{\I \theta}$ the topologies of the two spaces coincide, and hence we may use their properties interchangeably. For more information, see \cite[Remark~7.6]{Lions1972} that characterizes $H^s(\partial\Omega)$ for any smooth bounded domain $\Omega$ with the help of an eigensystem for the Laplace--Beltrami operator on $\partial \Omega$; observe also that $f_n$, $n \in \Z$, can be identified on $\partial D$ with an eigenfunction for the Laplace--Beltrami operator, which is essentially the second angular derivative, with the corresponding eigenvalue being $-n^2$. In particular, the elements of $H^s$ can be identified with continuous $2\pi$-periodic functions for $s>\tfrac{1}{2}$ due to the Sobolev embedding theorem, and the ($2\pi$-periodic) Dirac delta distribution belongs to $H^s$ for any $s < -\tfrac{1}{2}$. The latter could also be easily deduced by combining \eqref{eq:F_delta} and \eqref{eq:s_norm}.
\end{remark}

Let us complete this section by briefly considering the convergence of Fourier series in different spaces. The partial sums of a Fourier series  are defined by
\begin{equation*}
	S_N g := \sum_{\abs{n}\leq N} \hat{g}(n)f_n, \qquad N \in \N_0.
\end{equation*}
According to \cite[Theorem~5.2.1 and Section~5.3]{Saranen2002}, $S_N g \to g$ in $X$ as $N\to\infty$ for any of the choices $X\in \{\mathcal{D}_\textup{per},H^s, \mathcal{D}'_\textup{per}\}$, $s \in \R$. Moreover,
\begin{equation}
	\inner{f_n,f_m}_s = 2 \pi \, \underline{n}^{2s}\delta_{n,m},  \qquad n,m \in \Z, \label{eq:expinner} 
\end{equation}
meaning that the trigonometric monomials form an orthogonal basis of $H^s$ for any $s \in \R$.  It now follows immediately from  \eqref{eq:innerprod} that
\begin{equation} \label{eq:zero_meanP}
	P_\diamond g := \sum_{\abs{n}\geq 1} \hat{g}(n) f_n, \qquad g \in H^s,
\end{equation}
defines the orthogonal projection of $H^s$ onto $H^s_\diamond$ for any $s \in \R$. 

\section{Approximations based on equidistant interpolation points} \label{sec:approxuniform}

This section considers two techniques for approximating a periodic function based on its pointwise values on a uniform grid over a single period: trigonometric interpolation and forming a weighted linear combination of Dirac delta distributions placed at the quadrature points. The latter is needed when continuum current patterns are transformed into pointwise currents for the point electrode model, whereas the former gives a natural technique for extending (relative) point electrode measurements to smooth functions over the whole object boundary. 

\subsection{Interpolation by trigonometric polynomials} \label{sec:interptrig}

We follow \cite[Section~8.1--8.3]{Saranen2002}, always choosing an odd number of interpolation points to avoid certain asymmetry when it comes to the degrees of the employed trigonometric polynomials and Fourier coefficients.

Let $\mathcal{T}^M \subset \mathcal{D}_{\rm per}$ be the complex vector space of trigonometric polynomials of degree at most $M$ and let $\mathcal{T}_\diamond^M$ be its mean free subspace, that is,
\begin{equation*}
	\mathcal{T}^M := \mspan\{f_n\}_{\abs{n}\leq M}, \qquad \mathcal{T}_{\diamond}^M := \mspan\{f_n\}_{0 < \abs{n}\leq M}.
\end{equation*}
Due to the definition of the inner product of $H^s$ in \eqref{eq:innerprod}, the formula
\begin{equation} \label{eq:P_M}
	P_M g := \sum_{\abs{n} \leq M} \hat{g}(n)f_n, \qquad g\in H^s,
\end{equation}
defines the orthogonal projection of $H^s$ onto $\mathcal{T}^M$ for any $s\in\mathbb{R}$. Moreover,  $P_M|_{H^s_\diamond}$ is obviously the orthogonal projection of $H^s_\diamond$ onto  $\mathcal{T}_\diamond^M$, and a comparison with \eqref{eq:zero_meanP} easily leads to the conclusion that $P_\diamond P_M = P_M P_\diamond$ is the orthogonal projection of $H^s$ onto $\mathcal{T}_{\diamond}^M$. According to \cite[Theorem~8.2.1 and Exercise 8.2.1]{Saranen2002},
\begin{equation}
	\norm{\ident - P_M}_{\mathscr{L}(H^s,H^t)} = (1+M)^{t-s} \label{eq:ortprojest}
\end{equation}
for any $s,t\in\mathbb{R}$ such that $s\geq t$. 

As a side note, which is not directly connected to electrode models, we immediately obtain estimates for the discrepancy introduced when $\Upsilon(\sigma)$ is replaced by its natural finite-dimensional approximations with respect to a trigonometric basis of $\mathcal{T}_\diamond^{M}$.  Take note that, up to the numerical errors introduced by the employed forward solver, $\Upsilon(\sigma) P_M$ is a finite-dimensional operator one would naturally use in optimization-based reconstruction algorithms with finite amount of data in the framework of the CM. On the other hand, $P_M \Upsilon(\sigma) P_M$ is a matrix approximation for $\Upsilon(\sigma)$ with respect to a trigonometric basis of $\mathcal{T}_\diamond^{M}$, more suitable for direct reconstruction methods.
\begin{proposition}
	For $M\in\mathbb{N}$ and $s,t\in \mathbb{R}$ with $s\geq t$, 
	\begin{align*}
	\norm{\Upsilon(\sigma) - \Upsilon(\sigma) P_M}_{\mathscr{L}(H_\diamond^s,H_\diamond^t)} &\leq (1+M)^{t-s}\norm{\Upsilon(\sigma)}_{\mathscr{L}(H_\diamond^t)}, \\[1mm]
	\norm{\Upsilon(\sigma) - P_M \Upsilon(\sigma) P_M}_{\mathscr{L}(H^s_\diamond,H^t_\diamond)} &\leq (1+M)^{t-s}\left(\norm{\Upsilon(\sigma)}_{\mathscr{L}(H^t_\diamond)} + \norm{\Upsilon(\sigma)}_{\mathscr{L}(H^s_\diamond)}\right).
	\end{align*}
\end{proposition}
\begin{proof}
	The assertion is a direct consequence of \eqref{eq:ortprojest}:
	\begin{align*}
	\norm{\Upsilon(\sigma) - P_M \Upsilon(\sigma) P_M}_{\mathscr{L}(H^s_\diamond,H^t_\diamond)} &\leq \norm{\Upsilon(\sigma) (\ident - P_M)}_{\mathscr{L}(H^s_\diamond,H^t_\diamond)} + \norm{(\ident - P_M) \Upsilon(\sigma) P_M}_{\mathscr{L}(H^s_\diamond,H^t_\diamond)} \\[1mm]
	&\leq (1+M)^{t-s}\left(\norm{\Upsilon(\sigma)}_{\mathscr{L}(H^t_\diamond)} + \norm{\Upsilon(\sigma)}_{\mathscr{L}(H^s_\diamond)}\right),
	\end{align*}
	where we also used the fact $\norm{P_M}_{\mathscr{L}(H^s_\diamond)} = 1$ as the projection is orthogonal.
\end{proof}

Let us then consider trigonometric interpolation of $2\pi$-periodic functions based on their values at $2M+1$ equidistant points over a single period. We denote the interpolation points by
\begin{equation*}
	\theta_m := m \, \frac{2\pi}{2M+1}, \qquad  \abs{m}\leq M.
\end{equation*}
The trigonometric interpolation basis $\{\phi_m\}_{\abs{m} \leq M} \subset \mathcal{T}^M$ is then defined via
\begin{equation*}
	\phi_m(\theta) := \frac{1}{2M+1}\sum_{\abs{n}\leq M} f_n(\theta-\theta_m) = \frac{1}{2M+1}\sum_{\abs{n}\leq M} \e^{-\I n \theta_m} f_n(\theta), \qquad \abs{m}\leq M.
\end{equation*}
We are particularly interested in the following special properties of these functions:
\begin{equation} \label{eq:basisorth}
	\phi_m(\theta_j) = \frac{2M + 1}{2 \pi} \langle \phi_m, \phi_j\rangle_{L^2} = \delta_{j,m}, \qquad \abs{j},\abs{m}\leq M,
\end{equation}
which are straightforward consequences of the following identity that holds for any $k\in\mathbb{Z}$,
\begin{equation}
	\sum_{\abs{n}\leq M} \e^{\I n k \frac{2\pi}{2M+1}} = \begin{cases}
		0, & k \not\equiv 0\enskip (\text{mod } 2M+1), \\[1mm]
		2M+1, & k \equiv 0\enskip (\text{mod } 2M+1).
	\end{cases} \label{eq:expsum}
\end{equation}
The second line of \eqref{eq:expsum} is immediately evident, while the first one follows from the formula for a truncated geometric series. In particular, \eqref{eq:basisorth} guarantees that $\{\phi_m\}_{\abs{m} \leq M}$ are linearly independent, and thus form a basis for $\mathcal{T}^M$.

\begin{remark} \label{remark:pointe_D}
	If one interprets $\{ \theta_m\}_{\abs{m}\leq M}$ as angular coordinates, then
	\begin{equation} \label{eq:pointe_D}
		x_m := \e^{\I \theta_m}, \qquad \abs{m}\leq M,
	\end{equation}
	are the corresponding interpolation points (i.e.~point electrodes), on the boundary of the unit disk~$D$. In particular, $\partial D\setminus \{x_m\}_{\abs{m}\leq M}$ consists of $2M+1$ arcs of equal length.
\end{remark}

We are now ready to introduce the \emph{trigonometric interpolation operator} $Q_M : H^s \to \mathcal{T}_M$, $s>\tfrac{1}{2}$, defined in the natural manner
\begin{equation} \label{eq:Q_M}
	Q_M g := \sum_{\abs{m}\leq M} g(\theta_m)\phi_m. 
\end{equation}
As its name suggests $Q_M g(\theta_j) = g(\theta_j)$, $\abs{j} \leq M$, due to \eqref{eq:basisorth}. Take note that the definition \eqref{eq:Q_M} is unambiguous since any $g \in H^s$, $s > \tfrac{1}{2}$, is continuous by the Sobolev embedding theorem; see Remark~\ref{remark:lions}. As $\{\phi_m\}_{\abs{m}\leq M}$ is a basis for $\mathcal{T}^M$, it must hold that
\begin{equation}
	Q_M g = \sum_{\abs{m}\leq M} g(\theta_m)\phi_m = g, \qquad g \in \mathcal{T}^M, \label{eq:TMdecomp}
\end{equation}
because this is the only way to give $g$ as a linear combination of  $\{\phi_m\}_{\abs{m} \leq M}$ at the interpolation points by virtue of \eqref{eq:basisorth}. In other words, $Q_M$ is a non-orthogonal projection of $H^s$ onto $\mathcal{T}_M$.

Based on \eqref{eq:basisorth} and \eqref{eq:TMdecomp}, the periodic trapezoidal quadrature rule, with the quadrature points chosen to be the interpolation points $\{\theta_m\}_{\abs{m}\leq M}$, can be applied to exactly evaluate $L^2$-inner products of functions in $\mathcal{T}_M$:
\begin{equation}
	\inner{f,g}_{L^2} = \frac{2 \pi}{2M+1} \sum_{\abs{m}\leq M} f(\theta_m)\overline{g(\theta_m)}, \qquad f,g\in \mathcal{T}^M. \label{eq:quadrule}
\end{equation}
Moreover, by choosing $f = \phi_j$ and $g \equiv 1$ in \eqref{eq:quadrule} and employing \eqref{eq:basisorth} for one more time, we obtain
\begin{equation*}
	\int_{-\pi}^\pi \phi_j(\theta)\,\di\theta = \frac{2\pi}{2M+1}, \qquad \abs{j} \leq M.
\end{equation*}
Hence, 
\begin{equation}
	\sum_{\abs{m}\leq M} g(\theta_m) = 0 \qquad \text{if and only if} \qquad \int_{-\pi}^\pi Q_M g(\theta) \,\di\theta = \frac{2\pi}{2M+1} \sum_{\abs{m}\leq M} g(\theta_m) =  0. \label{eq:zeromean}
\end{equation}
In particular, the latter condition obviously holds for all $g$ in the mean free subspace $\mathcal{T}_\diamond^M$, and hence the former also holds for such $g$, as could have been directly verified using \eqref{eq:expsum} as well.

To complete this subsection, we recall the following estimate from \cite[Theorem~8.3.1]{Saranen2002}:
\begin{equation}
	\norm{\ident-Q_M}_{\mathscr{L}(H^s,H^t)} \leq C_s \bigl(\tfrac{1}{2}+M\bigr)^{t-s}, \qquad s>\frac{1}{2}, \ 0\leq t \leq s, \label{eq:interpolationbnd}
\end{equation}
where
\begin{equation} \label{eq:s_constant}
	C_s := \biggl(\sum_{n\geq 0} \underline{n}^{-2s}\biggr)^{1/2}.
\end{equation}
Comparing this to \eqref{eq:ortprojest}, one sees that $Q_M$ provides asymptotically in $M$, and up to a multiplicative constant, as good an approximation for the (embedding) identity operator as the orthogonal projection $P_M$, if the conditions in \eqref{eq:interpolationbnd} are satisfied.

\begin{remark} \label{remark:zeromean}
	According to \eqref{eq:zeromean}, the interpolant $Q_M g$ of a continuous $g$ has vanishing mean if and only if $[g(\theta_m)]_{\abs{m}\leq M} \in \C_\diamond^{2M+1}$. In particular, $[g(\theta_m)]_{\abs{m}\leq M} \in \C_\diamond^{2M+1}$ for $g \in \mathcal{T}_\diamond^M$ by virtue of \eqref{eq:TMdecomp}. Since the PEM for the unit disk $D$ with $2M+1$ electrodes placed at the equiangular points $\{x_m\}_{\abs{m}\leq M} \subset \partial D$ employs currents in $\mathbb{C}^{2M+1}_\diamond$, we may thus identify electrode current patterns with trigonometric polynomials in $\mathcal{T}_{\diamond}^M$. Indeed, the mapping
	\begin{equation} \label{eq:tildeQ}
		\hat{Q}_M:	I \mapsto \sum_{\abs{m}\leq M} I_{m} \phi_m
	\end{equation}
	defines a bijection between $\mathbb{C}^{2M+1}_\diamond$ and $\mathcal{T}_{\diamond}^M$. This gives a natural method for going back and forth between admissible electrode current patterns in $\C_\diamond^{2 M +1}$ and admissible continuum current patterns in $\mathcal{T}_\diamond^M \subset H^s_\diamond \simeq H^s_\diamond(\partial D)$. Take note that $\hat{Q}_M I = Q_M g$ if and only if $I = [g(\theta_m)]_{\abs{m}\leq M}$ for a continuous function $g$. Moreover,
        $$
        \big\| \hat{Q}_M I \big\|_{L^2}^2 = \frac{2 \pi}{2M + 1} \sum_{\abs{m}\leq M} |I_m|^2, \qquad I \in \C^{2 M +1},
        $$
        by virtue of \eqref{eq:basisorth}, and thus
        \begin{equation}
          \label{eq:hatQnorm}
        \| \hat{Q}_M \|_{\mathscr{L}(\C^{2 M +1}, L^2)} = \sqrt{\frac{2 \pi}{2M + 1}}
        \end{equation}
if $\C^{2 M +1}$ is equipped with the Euclidean norm.       
\end{remark}

\subsection{Pointwise approximation of smooth enough functions} \label{sec:pointwiseapprox}

In this subsection, we review how $2\pi$-periodic continuous functions can be approximated in weak Sobolev topologies by linear combinations of Dirac delta distributions supported at $\{\theta_m\}_{\abs{m}\leq M}$. To this end, we define a \emph{point evaluation operator} $F_M : H^s \to H^t$ for $s>\frac{1}{2}$ and $t < -\frac{1}{2}$ (cf.~Remark~\ref{remark:lions}): 
\begin{equation*}
	F_M g := \frac{2\pi}{2M+1} \sum_{\abs{m}\leq M}  g(\theta_m) \delta_{\theta_m},
\end{equation*}
where, as always, $\delta_{\theta_m}$ is identified with its periodic extension. It is worth noting that the multiplier $\tfrac{2\pi}{2M+1}$ is the mesh parameter corresponding to the employed grid $\{\theta_m\}_{\abs{m}\leq M}$.

\begin{lemma} \label{lemma:pointevalbnd}
	For $M\in \mathbb{N}$ and $s>\frac{1}{2}$, 
	\begin{equation*}
		\norm{\ident - F_M}_{\mathscr{L}(H^s, H^{-s})} \leq \bigl( 2 \, C_s + C_s^2 \bigr)\bigl(\tfrac{1}{2}+M \bigr)^{-s},
	\end{equation*}
    where the constant $C_s$ from \eqref{eq:s_constant} is independent of $M$.
\end{lemma}
\begin{proof}
	Let $f,g\in H^s$, $s>\tfrac{1}{2}$, and recall from Section~\ref{sec:Sobolev} that  $\inner{\cdot,\cdot}_{-s,s} : H^{-s}\times H^s \to \C$ denotes the sesquilinear dual bracket. By using the definition of $F_M$ and applying \eqref{eq:quadrule} to the interpolated functions $Q_M f$ and $Q_M g$, we deduce
	\begin{align*}
		\inner{(\ident-F_M)f,g}_{-s,s} &= \inner{f,g}_{L^2} - \frac{2 \pi}{2M+1}\sum_{\abs{m}\leq M} f(\theta_m)\overline{g(\theta_m)} \\
		&= \inner{f,g}_{L^2} - \inner{Q_M f, Q_M g}_{L^2} \\[1mm]
		&= \inner{f,(\ident-Q_M)g}_{L^2} + \inner{(\ident - Q_M)f,g}_{L^2} - \inner{(\ident- Q_M)f,(\ident-Q_M)g}_{L^2}.
	\end{align*} 
	Employing  \eqref{eq:interpolationbnd} and the continuity of the embedding $H^s\hookrightarrow L^2$ now yields
	\begin{align*}
		\abs{\inner{(\ident-F_M)f,g}_{-s,s}} &\leq \bigl(2 \,\norm{\ident- Q_M}_{\mathscr{L}(H^s,L^2)} + \norm{\ident- Q_M}_{\mathscr{L}(H^s,L^2)}^2\bigr)\norm{f}_s\norm{g}_s \\[1mm]
		&\leq \bigl(2\, C_s + C_s^2 \bigr) \bigl(\tfrac{1}{2}+M \bigr)^{-s}\norm{f}_s\norm{g}_s.
	\end{align*}
	Finally, taking the supremum over $f,g \in H^s$ with $\norm{f}_s = \norm{g}_s=1$ gives the sought for bound.
\end{proof}
When considering the point electrode model of EIT in the next section, it is mandatory that any linear combination of Dirac deltas defining a current pattern has vanishing mean, that is, the coefficients defining the linear combination must sum to zero. It is easy to check that $F_M g$ has zero mean if and only if $\sum_{\abs{m}\leq M} g(\theta_m) = 0$. In particular, we thus have $F_M|_{\mathcal{T}_\diamond^M} : \mathcal{T}_{\diamond}^M \to H^t_\diamond$, $t < -\tfrac{1}{2}$, based on the remark after \eqref{eq:zeromean}. To ensure the zero mean condition holds more generally, we need to combine $F_M$ with $G_M: H^s \to H^s$, $s > \tfrac{1}{2}$, defined by
\begin{equation} \label{eq:G_M}
	G_M g := g - \frac{1}{2M+1}\sum_{\abs{m}\leq M} g(\theta_m).
\end{equation}	
It is straightforward to check that $F_M G_M: H^s \to H^t_\diamond$ for any $s > \tfrac{1}{2}$ and $t < -\tfrac{1}{2}$. Moreover, $G_M$ is the identity operator when restricted to $\mathcal{T}_\diamond^M$, and thus $F_M G_M = F_M$ on $\mathcal{T}_\diamond^M$; see again the comment succeeding \eqref{eq:zeromean}.

The following lemma considers the approximation of identity by $F_M G_M$ on $H^s_\diamond$, $s>\frac{1}{2}$, i.e.,~on a space of admissible continuous current patterns for the continuum model of EIT. 
\begin{lemma} \label{lemma:pointevalbnd2}
	For $M\in \mathbb{N}$ and $s>\frac{1}{2}$, 
	\begin{equation*}
		\norm{\ident - F_MG_M}_{\mathscr{L}(H^s_\diamond, H^{-s}_\diamond)} \leq \bigl(1 + \sqrt{2}\,C_s \bigr) \bigl(2\,C_s + C_s^2\bigr)\bigl(\tfrac{1}{2}+M \bigr)^{-s},
	\end{equation*}
    where the constant $C_s$ from \eqref{eq:s_constant} is independent of $M$.
\end{lemma}
\begin{proof}
	We first demonstrate that proving the assertion boils down to writing a uniform estimate with respect to $M$ for the norm of $G_M: H^s \to H^s$, $s > \tfrac{1}{2}$. Indeed,
	\begin{align}\label{eq:mfree_est1}
		\norm{\ident - F_MG_M}_{\mathscr{L}(H^s_\diamond, H^{-s}_\diamond)} &= \sup_{f,g \in H^s_\diamond\setminus\{0\}} \frac{1}{\norm{f}_s \norm{g}_s} \abs{\inner{(\ident - F_MG_M)f, g}_{-s,s}} \nonumber  \\
		&= \sup_{f,g \in H^s_\diamond\setminus\{0\}} \frac{1}{\norm{f}_s \norm{g}_s} \abs{\inner{(\ident - F_M)G_M f, g}_{-s,s}} \nonumber \\[1mm]
		&\leq \norm{(\ident - F_M)G_M}_{\mathscr{L}(H^s, H^{-s})} \nonumber \\[1mm]
		&\leq \norm{\ident - F_M}_{\mathscr{L}(H^s, H^{-s})} \norm{G_M}_{\mathscr{L}(H^s)},
	\end{align}
	where the second step follows from the mean free function $g \in H^s_\diamond$ not seeing the constant $(\ident - G_M) f$, cf.~\eqref{eq:G_M}. Since the first term on the right-hand side of \eqref{eq:mfree_est1} can be handled by Lemma~\ref{lemma:pointevalbnd}, we only need to worry about the second one.

	A direct calculation gives for $g\in H^s$ with $s>\tfrac{1}{2}$:
	\begin{align}
		\norm{G_M g}_{s} &\leq \norm{g}_s + \frac{1}{2M+1}\sum_{\abs{m}\leq M}\norm{g(\theta_m)}_s \notag\\
		&= \norm{g}_s + \frac{\sqrt{2\pi}}{2M+1}\sum_{\abs{m}\leq M}\abs{g(\theta_m)} \notag\\
		&\leq \norm{g}_s + \sqrt{2 \pi} \sup_{\theta\in (-\pi,\pi]} \abs{g(\theta)} \notag\\
		&\leq \bigl( 1+ \sqrt{2} \, C_s \bigr)\norm{g}_s, \label{eq:Gbnd}
	\end{align}
	where the last step follows from the Sobolev embedding theorem; see \cite[Lemma~5.3.2]{Saranen2002} for the case of one-periodic functions. Combining \eqref{eq:mfree_est1} and \eqref{eq:Gbnd} with Lemma~\ref{lemma:pointevalbnd} completes the proof.
\end{proof}

\section{Equiangular point electrodes for the unit disk} \label{sec:unit_disk}

Let us return to the setting of Section~\ref{sec:forward} with $\Omega = D \subset \C$ being the unit disk; the case of a more general two-dimensional smooth and simply connected domain will be analyzed in Section~\ref{sec:general_domain} below, including considerations related to the CEM. In particular, we assume $\sigma\in L^\infty(D)$ satisfies the conditions in \eqref{eq:sigma_condition1} with $\Omega$ replaced by $D$. In what follows, we make the identification $\theta \leftrightarrow \e^{\I \theta}$, which means we can use $H^s$ and $H^s(\partial D)$, $s \in \R$, interchangeably; see Remark~\ref{remark:lions}.

Assume one would like to apply the continuum current pattern $f \in H^s_{\diamond}(\partial D)$, $s>\tfrac{1}{2}$, and measure the corresponding {\em relative potential} on the whole boundary $\partial D$, but due to practical restrictions the measurements need to be carried out with $2M+1$ infinitesimal electrodes at the (interpolation) points $\{\theta_m\}_{\abs{m}\leq M} \simeq \{x_m\}_{\abs{m}\leq M} \subset \partial D$ defined in \eqref{eq:pointe_D}. Our method for approximating the smooth output $\Upsilon(\sigma)f$ of the CM is as follows:
\begin{enumerate}[(i)]
	\item Introduce
	\begin{equation}\label{eq:explicitI}
		I = \frac{2 \pi}{2 M+1} \left( \begin{bmatrix}
    		f(\theta_{-M}) \\
    		\vdots \\
    		f (\theta_M)
    	\end{bmatrix}
		-\frac{1}{2 M +1} \sum_{\abs{m}\leq M}f(\theta_m) \, \mathbf{1} \right) \in \C_\diamond^{2 M+1}
	\end{equation}
	as the electrode current pattern for the PEM.
	\item Perform the measurements of the PEM to retrieve $U = \Upsilon_{\rm PEM}(\sigma) I \in \C_\diamond^{2 M +1}$.
	\item Build an approximation $g$ for the continuum boundary potential $\Upsilon(\sigma)f$ via trigonometric interpolation, that is, $g = \hat{Q}_M U \in \mathcal{T}_\diamond^M$, where the bijection $\hat{Q}_M$ is defined by \eqref{eq:tildeQ}.
\end{enumerate}
Although it is not necessarily completely evident, another way of writing the above approximation procedure is
\begin{equation}\label{eq:equivalence}
	g = P_\diamond Q_M \Upsilon(\sigma) F_M G_M f.
\end{equation}
Let us clarify this claim. First of all, it is easy to check that
\begin{equation*}
	F_M G_M f = f_I \in H^{-s}_\diamond
\end{equation*}
for $I$ from \eqref{eq:explicitI} and the corresponding $f_I$ defined as in \eqref{eq:f_I} with $\Omega$ replaced by $D$ and $\{y_m \}_{\abs{m}\leq M}$ by $\{x_m \}_{\abs{m}\leq M}$. Thus $U \in \C_\diamond^{2 M +1}$ equals the vector obtained by evaluating $\Upsilon(\sigma) f_I$ at the point electrodes up to the addition of a multiple of $\mathbf{1} \in \C^{2 M +1}$ (cf.~\eqref{eq:Upsilon}). It is easy to check that this component in the direction of  $\mathbf{1}$ only affects the component of the trigonometric polynomial $\hat{Q}_M U \in \mathcal{T}^M$ in the direction of the constant function, meaning in particular that
\begin{equation*}
	Q_M \Upsilon(\sigma) F_M G_M f = Q_M \Upsilon(\sigma) f_I =  \hat{Q}_M U + c = g + c
\end{equation*}
for some constant $c \in \C$. Taking the orthogonal projection $P_\diamond$ onto the mean free trigonometric polynomials thus proves \eqref{eq:equivalence}.

To summarize, considering the accuracy of the above introduced procedure for mimicking CM measurements for the unit disk, with a finite number of electrodes $\{x_m\}_{\abs{m}\leq M}$ modeled by the PEM, is equivalent to proving estimates for the discrepancy between the CM measurement map $\Upsilon(\sigma)$ and its modified version $P_\diamond Q_M \Upsilon(\sigma) F_M G_M$.
\begin{theorem} \label{thm:est1}
	For $M\in \mathbb{N}$ and $s>\frac{1}{2}$, 
	\begin{equation}
		\norm{\Upsilon(\sigma) - P_\diamond Q_M \Upsilon(\sigma) F_M G_M}_{\mathscr{L}(H^s_\diamond, L^2_\diamond)} \leq C\bigl(\tfrac{1}{2}+M \bigr)^{-s}, \label{eq:finalest}
	\end{equation}
    where the positive constant $C=C(s,\sigma)$ is independent of $M$.
\end{theorem}
\begin{proof}
	Let us consider an arbitrary $s > \tfrac{1}{2}$ and $f\in H^s_\diamond$. Since $\Upsilon(\sigma)\in \mathscr{L}(H^{-s}_\diamond,H^s_\diamond)$ due to \eqref{eq:Tbound} and $(\ident - F_M G_M)\in \mathscr{L}(H^{s}_\diamond,H^{-s}_\diamond)$ by Lemma~\ref{lemma:pointevalbnd2}, the Sobolev embedding theorem gives
	\begin{equation}
		\sup_{\theta\in (-\pi,\pi]}\abs{\Upsilon(\sigma)(\ident - F_M G_M) f(\theta)} \leq C(s,\sigma) \bigl(\tfrac{1}{2}+M \bigr)^{-s} \norm{f}_s. \label{eq:diffNDest}
	\end{equation}
	Combining \eqref{eq:quadrule} and \eqref{eq:diffNDest}, we thus get
	\begin{align}
		\norm{Q_M \Upsilon(\sigma) (\ident-F_M G_M)f}_{L^2} &= \biggl(\frac{2\pi}{2M+1}\sum_{\abs{m}\leq M} \abs{\Upsilon(\sigma) (\ident - F_M G_M)f(\theta_m)}^2\biggr)^{1/2} \notag\\
		&\leq C(s,\sigma) \bigl(\tfrac{1}{2}+M\bigr)^{-s}\norm{f}_s. \label{eq:est1}
	\end{align}
	As $P_\diamond \Upsilon(\sigma) f = \Upsilon(\sigma) f \in H^s_\diamond$ and $\norm{P_\diamond}_{\mathscr{L}(L^2)} = 1$ due to $P_\diamond$ being an orthogonal projection, we have
	\begin{align} \label{eq:help_est}
		\norm{\Upsilon(\sigma) f - P_\diamond Q_M \Upsilon(\sigma) F_MG_M f}_{L^2} & \leq \norm{P_\diamond(\ident - Q_M)\Upsilon(\sigma) f}_{L^2} + \norm{P_\diamond Q_M \Upsilon(\sigma) (\ident - F_M G_M) f}_{L^2}  \nonumber \\[1mm]
		&\leq C_s\bigl(\tfrac{1}{2}+M\bigr)^{-s} \norm{\Upsilon(\sigma) f}_s + C(s,\sigma)\bigl(\tfrac{1}{2}+M\bigr)^{-s}\norm{f}_s  \\[1mm]
		&\leq C(s,\sigma) \bigl(\tfrac{1}{2}+M\bigr)^{-s}\norm{f}_s, \nonumber
	\end{align}
	where in the second step we also used \eqref{eq:interpolationbnd} and \eqref{eq:est1}, and the final inequality is a consequence of~\eqref{eq:Tbound}. This completes the proof.
\end{proof}

\begin{remark}
	Theorem \ref{thm:est1} remains valid if $Q_M$ is replaced by the orthogonal projection $P_M$. Indeed, this claim can be proved by simply substituting $P_M$ for $Q_M$ in the proof of Theorem~\ref{thm:est1} accompanied by three simple modifications: replacing $\eqref{eq:diffNDest}$ by a (weaker) bound for the $L^2$ norm of $\Upsilon(\sigma)(\ident - F_M G_M) f$, resorting to the boundedness of $P_M: L^2 \to L^2$ in \eqref{eq:est1}, and using \eqref{eq:ortprojest} instead of \eqref{eq:interpolationbnd} in \eqref{eq:help_est}. Be that as it may, one can only approximately apply $P_M$ in connection to PEM measurements by employing $\{\theta_m\}_{\abs{m}\leq M}$ as quadrature nodes for the (periodic) trapezoidal rule when estimating the needed Fourier coefficients (cf.~\eqref{eq:P_M}), which corresponds to computing a discrete Fourier transformation. It is straightforward to check that such an approximative approach to applying $P_M$ is actually analogous to using the interpolation operator $Q_M$ to begin with. 
\end{remark}

If one restricts the attention to a finite family of preselected trigonometric current patterns, Theorem~\ref{thm:est1} leads to convergence of arbitrarily high order with respect to $M$ in $L^2$.  
\begin{corollary}
  \label{cor:exponential}
	For $M \in \mathbb{N}$, $n\in\mathbb{Z}\setminus\{0\}$, and $s>\frac{1}{2}$,
	\begin{equation*}
		\norm{(\Upsilon(\sigma) - P_\diamond Q_M \Upsilon(\sigma) F_M G_M)f_n}_{L^2} \leq C \left(\frac{\abs{n}}{\tfrac{1}{2}+M}\right)^s,
	\end{equation*}
	where the positive constant $C=C(s,\sigma)$ is independent of $n$ and $M$.
\end{corollary}
\begin{proof}
	The result is a direct consequence of the identity $\norm{f_n}_s = \sqrt{2\pi}\abs{n}^s$ (cf.~\eqref{eq:expinner}) and Theorem~\ref{thm:est1}.
\end{proof}

Moreover, the approximating operator $P_\diamond Q_M \Upsilon(\sigma) F_M G_M$ inherits certain symmetry and monotonicity properties from $\Upsilon(\sigma)$, if one once again restricts the attention to trigonometric polynomials.
\begin{proposition}
	If $v,w\in \mathcal{T}^M_\diamond$ and $s>\frac{1}{2}$,
	\begin{equation*}
		\inner{v,P_\diamond Q_M \Upsilon(\sigma) F_M G_Mw}_{L^2} = \inner{F_M v, \Upsilon(\sigma)F_M w}_{-s,s}.			
	\end{equation*}
\end{proposition}
\begin{proof}
	Since $v,w\in\mathcal{T}^M_\diamond$, it holds that $G_M w = w$ and we may also remove $P_\diamond$ from the inner product on the left-hand side of the assertion. By denoting $\tilde{w} = \Upsilon(\sigma)F_M w \in H^s$, what remains is 
	\begin{equation*}
		\inner{v,Q_M\tilde{w}}_{L^2} = \frac{2\pi}{2M+1}\sum_{\abs{m}\leq M} v(\theta_m)\overline{\tilde{w}(\theta_m)} = \inner{F_M v,\tilde{w}}_{-s,s},
	\end{equation*}
	where we used the definitions of $Q_M$ and $F_M$ combined with \eqref{eq:quadrule}.
\end{proof}

\section{Conformally mapped electrodes for a simply connected domain} \label{sec:general_domain}
In this section, the first aim is to generalize Theorem~\ref{thm:est1} to the case of a general smooth bounded simply connected domain, which is achieved with the help of the Riemann mapping theorem. Subsequently, we transfer the obtained result to the framework of the CEM by resorting to the material in \cite{Hanke2011b}.

As in Section~\ref{sec:forward},  let $\Omega$ be a smooth, bounded, and simply connected domain and assume $\sigma\in L^\infty(\Omega)$ satisfies the conditions in \eqref{eq:sigma_condition1}. According to the Riemann mapping theorem, there exists a bijective conformal mapping $\Phi: \Omega \to D$, whose restriction to $\partial \Omega$ defines a $C^\infty$-diffeomorphism between $\partial \Omega$ and $\partial D$ \cite{Pommerenke1992}. The inverse of $\Phi$ is denoted by $\Psi := \Phi^{-1}$, and its complex derivative is~$\Psi'$. We introduce $2M+1$ point electrodes (or electrode midpoints) on $\partial \Omega$ via $y_m := \Psi(x_m)$, $\abs{m}\leq M$, with the interpolation points $\{x_m\}_{\abs{m}\leq M} \subset \partial D$ defined by \eqref{eq:pointe_D}. 

\subsection{PEM and a general domain}
Analogously to the case of $D$ in the previous section, the aim is to drive a continuum current pattern $f \in H^s_{\diamond}(\partial \Omega)$, $s>\tfrac{1}{2}$, through $\partial \Omega$ and measure the corresponding relative boundary potential everywhere on $\partial \Omega$. However, due to practical restrictions, the measurements are performed with $2M+1$ infinitesimal electrodes at the positions $\{y_m\}_{\abs{m}\leq M} \subset \partial \Omega$. Our method for approximating the smooth output $\Upsilon(\sigma)f$ of the CM is as follows:
\begin{enumerate}[(i)]
\item Introduce
	\begin{equation}\label{eq:explicitIO}
	I = \frac{2 \pi}{2 M+1} \left( 
	\begin{bmatrix}
    	\abs{\Psi'(x_{-M})} \, f(y_{-M}) \\
    	\vdots \\
    	\abs{\Psi'(x_{M})} \, f(y_{M})
  	\end{bmatrix}
	-\frac{1}{2 M +1} \sum_{\abs{m}\leq M} \abs{\Psi'(x_{m})}\,f(y_{m}) \, \mathbf{1} \right) \in \C_\diamond^{2 M+1}
	\end{equation}
	as the electrode current pattern for the PEM.
	\item Perform the measurements of the PEM to retrieve $U = \Upsilon_{\rm PEM}(\sigma) I \in \C_\diamond^{2 M +1}$.
	\item Build the trigonometric interpolant
	\begin{equation}\label{eq:tildeg}
		\tilde{g} = \hat{Q}_M U \in \mathcal{T}_\diamond^M
	\end{equation}
	on $\partial D$ with respect to the equidistant interpolation points  $\{x_m\}_{\abs{m}\leq M}$.
	\item Form an approximation for the continuum boundary potential as $g = \tilde{g} \circ \Phi + c$, where $c\in\C$ is chosen so that the smooth function $g$ has vanishing mean.
\end{enumerate}
The above construction obviously defines a bounded linear operator
\begin{equation*}
	\Upsilon_{M}(\sigma):
	\begin{cases}
		f \mapsto g, \\[1mm]
		H_\diamond^s(\partial \Omega) \to H_\diamond^t(\partial \Omega),
	\end{cases}
\end{equation*}
for any $s > \tfrac{1}{2}$ and $t \in \R$. According to the following theorem, this newly introduced operator gives an approximation for the measurement map of the CM.

\begin{theorem} \label{thm:main}
	For $M\in \mathbb{N}$ and $s>\frac{1}{2}$, 
	\begin{equation*}
		\norm{\Upsilon(\sigma) - \Upsilon_{M}(\sigma)}_{\mathscr{L}(H^s_\diamond(\partial \Omega), L^2_\diamond(\partial \Omega))} \leq C \bigl(\tfrac{1}{2}+M \bigr)^{-s}, 
	\end{equation*}
	where the positive constant $C = C(s,\sigma, \Phi)$ is independent of $M$.
\end{theorem}
\begin{proof}
	Let $s>\tfrac{1}{2}$ and $f \in H^{s}_\diamond(\partial \Omega)$ be arbitrary but fixed, and let $I \in \C_\diamond^{2M+1}$ be defined by \eqref{eq:explicitIO}.  Moreover, let $u_\sigma \in H^1(\Omega)/\C$ be the solution to \eqref{eq:PDE} with  the Neumann boundary condition \eqref{eq:Neumann}. Then $\tilde{u}_{\tilde{\sigma}} := u_\sigma \circ \Psi \in H^1(D)/\C$ is the unique solution of
	\begin{equation}\label{eq:PDE_D2}
		\nabla \cdot (\tilde{\sigma} \nabla \tilde{u}) = 0 \quad \text{in } D, \qquad \frac{\partial \tilde{u}}{\partial \nu} = \tilde{f}  \quad \text{on } \partial D,
	\end{equation}
	where $\tilde{f} := \abs{\Psi'} (f \circ \Psi)$ and $\tilde{\sigma} = \sigma \circ \Psi$; see,~e.g.,~\cite[Lemma~4.1 and Remark~4.1]{Hyvonen2018}. On the other hand, if $w_\sigma \in H^r(\Omega)/\C$, $r < 1$, is the solution to the PEM forward problem 
	\begin{equation}\label{eq:PDE_O}
		\nabla \cdot (\sigma \nabla w) = 0 \quad \text{in } \Omega, \qquad \frac{\partial w}{\partial \nu} =  \sum_{\abs{m}\leq M} I_m  \delta_{y_m} \quad \text{on } \partial \Omega,
	\end{equation}
	then $\tilde{w}_{\tilde{\sigma}} := w_\sigma \circ \Psi\in H^r(D)/\C$, $r < 1$, is the unique solution to
	\begin{equation}\label{eq:PDE_D3}
		\nabla \cdot (\tilde{\sigma} \nabla \tilde{w}) = 0 \quad \text{in } D, \qquad \frac{\partial \tilde{w}}{\partial \nu} =  \sum_{\abs{m}\leq M} I_m  \delta_{x_m} \quad \text{on } \partial D,
	\end{equation}
	because Neumann boundary values involving Dirac delta distributions transform naturally under conformal maps; see the proof of \cite[Theorem~3.2]{Hakula2011}. The potentials $u_1$, $\tilde{u}_1$, $w_1$ and $\tilde{w}_1$ corresponding to the unit conductivity are defined by setting $\sigma \equiv 1$ in the preceding formulas.

	Let $\tilde{\Upsilon}(\tilde{\sigma}): H_\diamond^{-t}(\partial D) \to H_\diamond^{t}(\partial D)$, $t \in \R$, be the relative CM boundary map corresponding to $D$ and $\tilde{\sigma}$ (cf.~\eqref{eq:T}). Obviously,
	\begin{equation*}
		\tilde{\Upsilon}(\tilde{\sigma}) \tilde{f} = (\tilde{u}_{\tilde{\sigma}} - \tilde{u}_1)|_{\partial D} + c_1
	\end{equation*}
	and
	\begin{equation*}
		\tilde{\Upsilon}(\tilde{\sigma}) F_M G_M \tilde{f} = (\tilde{w}_{\tilde{\sigma}} - \tilde{w}_1)|_{\partial D} + c_2,
	\end{equation*}
	where $c_1, c_2 \in \C$ are chosen so that $\tilde{\Upsilon}(\tilde{\sigma}) \tilde{f}$ and $\tilde{\Upsilon}(\tilde{\sigma}) F_M G_M \tilde{f}$ have vanishing means.
	Comparing \eqref{eq:explicitIO} and \eqref{eq:PDE_D3} with the construction in Section~\ref{sec:unit_disk}, it easily follows that $\tilde{g}$ from \eqref{eq:tildeg} satisfies
	\begin{equation*}
		\tilde{g} = P_\diamond Q_M \tilde{\Upsilon}(\tilde{\sigma}) F_M G_M \tilde{f}.
	\end{equation*}
	In consequence, Theorem~\ref{thm:est1} yields
	\begin{align}\label{eq:conform}
		\norm{(\tilde{u}_{\tilde{\sigma}} - \tilde{u}_1) - \tilde{g}}_{L^2(\partial D) / \C} &=  \norm{(\tilde{\Upsilon}(\tilde{\sigma}) - P_\diamond Q_M \tilde{\Upsilon}(\tilde{\sigma}) F_M G_M) \tilde{f}}_{L^2(\partial D) / \C} \nonumber \\[1mm]
		&\leq C(s, \tilde{\sigma}) \bigl(\tfrac{1}{2}+M \bigr)^{-s} \norm{\tilde{f}}_{H^s(\partial D)} \nonumber \\[1mm]
		& \leq C(s, \sigma, \Phi) \bigl(\tfrac{1}{2}+M \bigr)^{-s} \norm{f}_{H^s(\partial \Omega)},
	\end{align}
	where the final step directly follows for integer $s$ from the smoothness of $\Psi|_{\partial D}: \partial D \to \partial \Omega$ and the definitions of the norms $\norm{\cdot}_{H^s(\partial D)}$ and $\norm{\cdot}_{H^s(\partial \Omega)}$. For non-integer $s$, it can be deduced,~e.g.,~via interpolation of Sobolev spaces~\cite{Lions1972}.
	
	Since $u_\sigma = \tilde{u}_{\tilde{\sigma}} \circ \Phi$ and $u_1 = \tilde{u}_{1} \circ \Phi$ modulo additive constants, we may write
	\begin{align*}
		\norm{(\Upsilon(\sigma) - \Upsilon_M(\sigma)) f}_{L^2(\partial \Omega)/\C} &= \norm{(u_\sigma - u_1) - g}_{L^2(\partial \Omega)/\C} \\[1mm]
		&= \big\| \bigl( (\tilde{u}_{\tilde{\sigma}} - \tilde{u}_1) - \tilde{g} \bigr) \circ \Phi \big\|_{L^2(\partial \Omega)/\C} \\[1mm]
		&\leq C(\Phi) \norm{(\tilde{u}_{\tilde{\sigma}} - \tilde{u}_1) - \tilde{g}}_{L^2(\partial D) / \C} \\[1mm]
		&\leq C(s, \sigma, \Phi) \bigl(\tfrac{1}{2}+M \bigr)^{-s} \norm{f}_{H^s(\partial \Omega)},
	\end{align*}
	where the last step is simply \eqref{eq:conform}. Since the mean free representative of an equivalence class in $L^2(\partial \Omega)/\C$ realizes the corresponding quotient norm and $(\Upsilon(\sigma) - \Upsilon_M(\sigma)) f \in L^{2}_\diamond(\partial \Omega)$ by the definitions of $\Upsilon(\sigma)$ and $\Upsilon_M(\sigma)$, we in fact have
	\begin{equation*}
		\norm{(\Upsilon(\sigma) - \Upsilon_M(\sigma)) f}_{L^2(\partial \Omega)} \leq C(s, \sigma, \Phi) \bigl(\tfrac{1}{2}+M \bigr)^{-s} \norm{f}_{H^s(\partial \Omega)}.
	\end{equation*}
	As $f \in H^s_{\diamond}(\partial \Omega)$ was chosen arbitrarily, this completes the proof.
\end{proof}

\subsection{CEM and a general domain}

To complete this section, let us consider the discrepancy between the CM and the CEM. We choose the midpoints of the finite-sized electrodes $\{ E_{m} \}_{\abs{m}\leq M}$ for the CEM to be the conformal images $\{ y_m \}_{\abs{m} \leq M}$ of the equidistant interpolation points $\{ x_m \}_{\abs{m} \leq M}$ on $\partial D$. Moreover, the lengths of the electrode patches along $\partial \Omega$ are assumed to satisfy
\begin{equation} \label{eq:el_length}
	c_\textup{E} d \leq \abs{E_m} \leq C_\textup{E} d, \qquad \abs{m}\leq M,
\end{equation}
where $c_\textup{E}, C_\textup{E} > 0$ are positive constants independent of $M$, $d \in (0, d_0)$ is a parameter controlling the widths of the electrodes, and $d_0 = d_0(M,\Omega, \Phi)> 0$ is chosen such that the electrodes cannot overlap. We define the mapping
\begin{equation}
  \label{eq:hatUpsilon}
	\hat{\Upsilon}_{M}(\sigma):
	\begin{cases}
		f \mapsto g, \\[1mm]
		H_\diamond^s(\partial \Omega) \to H_\diamond^t(\partial \Omega), \quad s > \tfrac{1}{2}, \ t \in \R,
	\end{cases}
\end{equation}
by replacing $\Upsilon_{\rm PEM}(\sigma)$ with $\Upsilon_{\rm CEM}(\sigma)$ defined by \eqref{eq:Upsilon_CEM} in the construction of $\Upsilon_M(\sigma)$. In other words, the relative measurements of the realistic CEM are used when forming an approximation for $\Upsilon(\sigma)$ instead of those of the less practical PEM. 

In particular, by choosing $d = \mathcal{O}((\frac{1}{2}+M)^{-s/2})$, i.e.\ suitably shrinking the electrodes when more electrodes are introduced, we obtain the same estimate for the CEM as for the PEM in Theorem~\ref{thm:main}. This is an immediate consequence of the following result.

\begin{corollary} \label{cor:main}
	For $M\in \mathbb{N}$ and $s>\frac{1}{2}$, 
	\begin{equation*}
		\norm{\Upsilon(\sigma) - \hat{\Upsilon}_{M}(\sigma)}_{\mathscr{L}(H^s_\diamond(\partial \Omega), L^2_\diamond(\partial \Omega))} \leq C \bigl(\bigl(\tfrac{1}{2}+M \bigr)^{-s} + d^2\bigr), 
	\end{equation*}
	where the positive constant $C = C(s,\sigma,z,\Phi,\Omega)$ is independent of $M$ and $d \in (0, d_0)$.
\end{corollary}

\begin{proof}
	To begin with, we write
	\begin{align} \label{eq:final_estim}
		\norm{\Upsilon(\sigma) - \hat{\Upsilon}_{M}(\sigma)}_{\mathscr{L}(H^s_\diamond(\partial \Omega), L^2_\diamond(\partial \Omega))} &\leq \norm{\Upsilon(\sigma) - \Upsilon_{M}(\sigma)}_{\mathscr{L}(H^s_\diamond(\partial \Omega), L^2_\diamond(\partial \Omega))} \notag \\[1mm]
		& \quad \quad + \norm{\Upsilon_M(\sigma) - \hat{\Upsilon}_{M}(\sigma)}_{\mathscr{L}(H^s_\diamond(\partial \Omega), L^2_\diamond(\partial \Omega))} \\[1mm]
		& \leq C(s,\sigma, \Phi) \bigl(\tfrac{1}{2}+M \bigr)^{-s}  + \norm{\Upsilon_M(\sigma) - \hat{\Upsilon}_{M}(\sigma)}_{\mathscr{L}(H^s_\diamond(\partial \Omega), L^2_\diamond(\partial \Omega))} \notag
	\end{align}
	due to Theorem~\ref{thm:main}. By virtue of \eqref{eq:hatQnorm} and Theorem~\ref{thm:M-point} (cf.~\cite[Corollary~2.1]{Hanke2011b}),
	\begin{align} \label{eq:final_estim2}
		\norm{( \Upsilon_M(\sigma) - \hat{\Upsilon}_{M}(\sigma)) f}_{L^2_\diamond(\partial \Omega)} & =  \norm{ \hat{Q}_M (\Upsilon_{\rm PEM}(\sigma) - \Upsilon_{\rm CEM}(\sigma)) I \circ \Phi}_{L^2(\partial \Omega) /\C} \nonumber \\[1mm]
		&\leq C(\sigma,z,\Phi,\Omega) M^{1/2} d^2 \abs{I},
	\end{align}
	where $I$ is defined by \eqref{eq:explicitI} and $\abs{I}$ denotes its Euclidean norm. Applying the Sobolev embedding theorem to \eqref{eq:explicitI}, it straightforwardly follows that
	\begin{equation*}
		\abs{I}^2 \leq \frac{C(s, \Phi)}{2M + 1} \norm{f}_{H^s(\partial \Omega)}^2.
	\end{equation*}
	Combining this with \eqref{eq:final_estim} and \eqref{eq:final_estim2} completes the proof.
\end{proof}

\begin{remark}
	An advantage of using the estimate for the PEM from Theorem~\ref{thm:main} to obtain the result in Corollary~\ref{cor:main} for the CEM is that we only use the conformal mapping to determine the midpoints of the electrodes on  $\partial \Omega$. If one were to conformally map CEM electrodes from $\partial D$ to $\partial\Omega$, there would be unwanted shrinking and stretching of the corresponding electrodes on $\partial \Omega$.
\end{remark}

\section{Numerical examples}
\label{sec:numer}
This section presents two numerical examples that verify the convergence rates in Corollaries~\ref{cor:main} and \ref{cor:exponential}, respectively, for simple conductivities in the unit disk. 

\begin{example}
  \label{ex:first}
  Let our domain of interest be the unit disk, i.e.\ $\Omega = D$, and assume that the conductivity inside $D$ is of the form
  \begin{equation}
    \label{eq:cond_con}
  \varsigma :=
\left\{
\begin{array}{ll}
  1 & \text{in } D \setminus \overline{D_{0,R}}, \\[1mm]
  \kappa & \text{in } D_{0,R}
\end{array}
\right.
\end{equation}
where $\kappa \in \R_+$ is a positive constant and $D_{0,R}$ is the origin-centered open disk of radius $0 < R < 1$. It is obvious that $\varsigma$ satisfies the conditions \eqref{eq:sigma_condition1}.

Let $\Phi: D \to D$ be the identity map, which means that the midpoints of the electrodes $y_m = \Psi(x_m) = x_m$, $|m| \leq M$, are the interpolation points defined by \eqref{eq:pointe_D}.  We introduce two sequences of CEM-based approximating operators $\hat{\Upsilon}_M(\varsigma)$ defined by \eqref{eq:hatUpsilon}; they correspond to different choices for the dependence of the common width $d_M := |E_{-M}| = \cdots = |E_{M}|$ of the employed electrodes on their number:
\begin{equation}
  \label{eq:elec1}
d_M = \frac{\pi}{2M +1} = \mathcal{O}(M^{-1}), \qquad M \in \N,
\end{equation}
and
\begin{equation}
  \label{eq:elec2}
d_M = \frac{\pi}{M(2M + 1)} = \mathcal{O}(M^{-2}), \qquad M \in \N.
\end{equation}
In the first case, the $2M +1$ electrodes always cover fifty percent of the object boundary, whereas the area covered by the electrodes is inversely proportional to $M$ in the latter construction. According to Corollary~\ref{cor:main}, these choices lead to the asymptotic convergence rate
\begin{equation} \label{eq:hypothesis}
	\norm{\Upsilon(\varsigma) - \hat{\Upsilon}_{M}(\varsigma)}_{\mathscr{L}(H^s_\diamond(\partial D), L^2_\diamond(\partial D))} = \mathcal{O}(M^{-s}),
\end{equation}
with $s = 2$ for \eqref{eq:elec1} and $s=4$ for \eqref{eq:elec2}.

\begin{figure}[htb]
  \includegraphics[width=70mm]{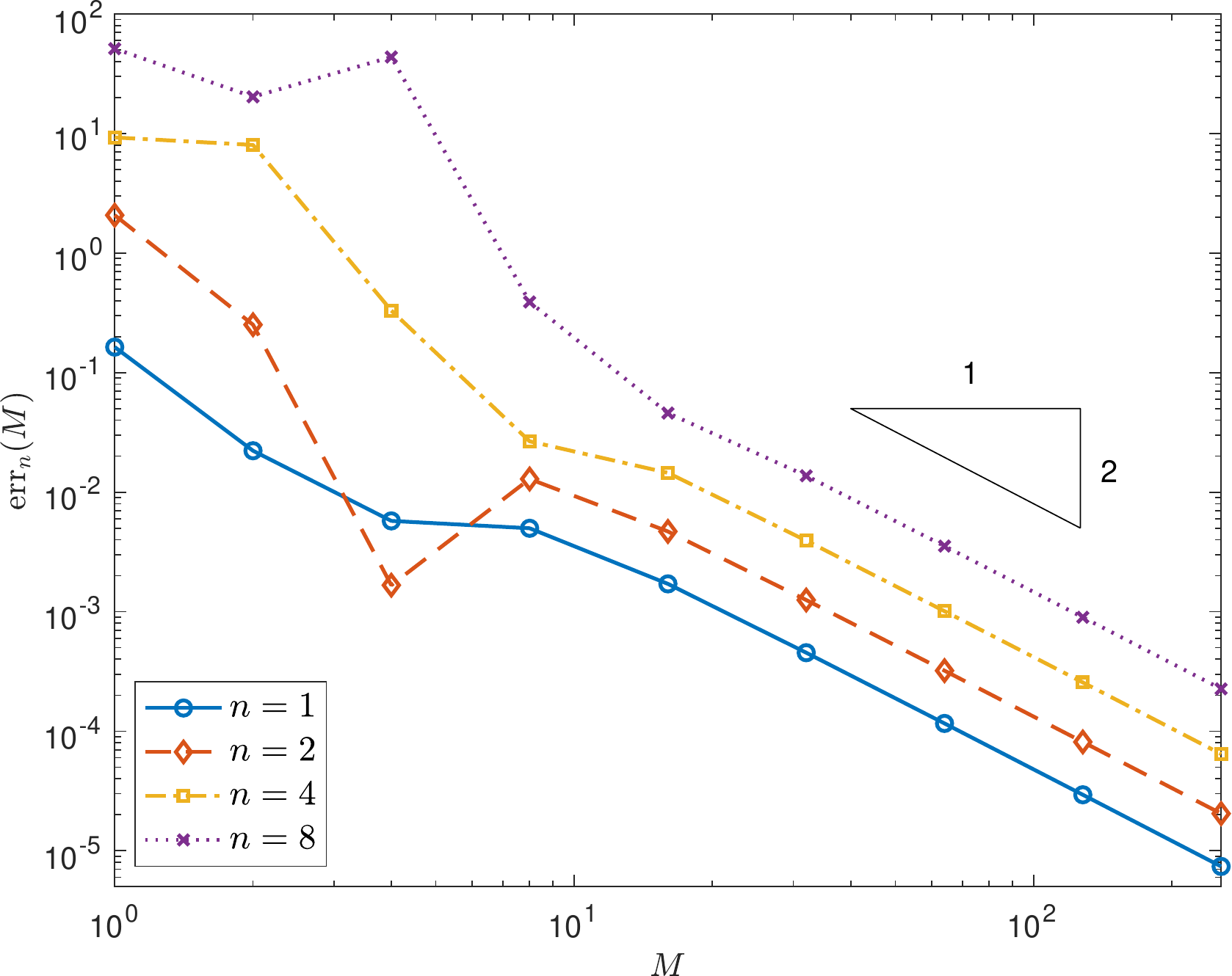} \qquad \includegraphics[width=70mm]{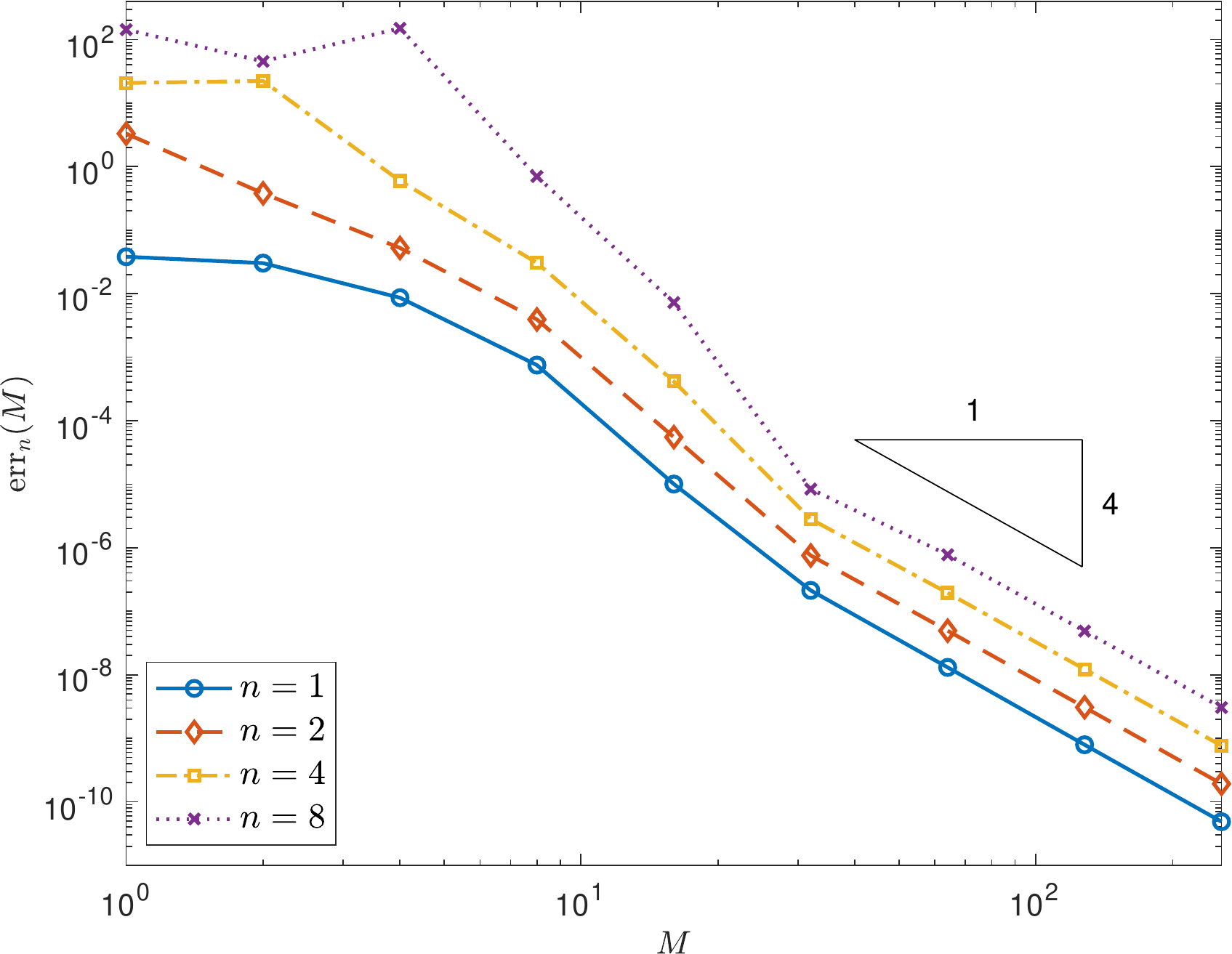}
  \caption{The relative $L^2(\partial D)$ discrepancy ${\rm err}_n(M)$ in \eqref{eq:relative_error} as a function of $M$ for $R=0.9$, $\kappa = 0.5$, and $z_{-M} = \cdots = z_M = 1$. Circles and solid blue line: $n=1$. Diamonds and dashed red line: $n=2$. Boxes and dash-dotted yellow line: $n=4$. Crosses and dotted magenta line: $n =8$. Left: electrode widths given by \eqref{eq:elec1}. Right:  electrode widths given by \eqref{eq:elec2}.} \label{fig:figure1}
\end{figure}

Figure~\ref{fig:figure1} validates \eqref{eq:hypothesis} numerically for a few Fourier basis functions $f_n$, $R = 0.9$, $\kappa = 0.5$, and the contact resistances set to $z=1$ on all electrodes; observe that we once again identify points on $\partial D$ with their polar angles. To be more precise, instead of exactly mimicking the left-hand side of \eqref{eq:hypothesis}, i.e.~including $\norm{f_n}_{H^s(\partial D)}$ in the denominator as in the definition of the operator norm, we plot the relative $L^2(\partial D)$ errors
\begin{equation} \label{eq:relative_error}
{\rm err}_n(M) := \frac{\norm{\bigl(\Upsilon(\varsigma) - \hat{\Upsilon}_{M}(\varsigma)\bigr) f_n}_{L^2(\partial D)}}{\norm{\Upsilon(\varsigma) f_n}_{L^2(\partial D)}}, \qquad n=1, 2, 4, 8,
\end{equation}
as functions of $M$ for both choices \eqref{eq:elec1} and \eqref{eq:elec2}. For a fixed $n$, ${\rm err}_n$ obviously satisfies a same type of an estimate as \eqref{eq:hypothesis}. Moreover, ${\rm err}_n$ is the inverse of the ratio between the desired output signal and the approximation error resulting from not being able to apply and measure continuum currents and potentials on $\partial D$, and thus it can be interpreted as a certain kind of a noise-to-signal ratio. According to Figure~\ref{fig:figure1}, the asymptotic convergence rates are as predicted by \eqref{eq:hypothesis} for both \eqref{eq:elec1} and \eqref{eq:elec2}; take note that infinitely small electrodes would actually result in  ${\rm err}_n$ converging faster than any negative power of $M$, that is, for a fixed spatial Fourier frequency the discrepancy between the CEM and the PEM can be considered to be the main source of asymptotic error in Figure~\ref{fig:figure1} (cf.~Example~\ref{ex:second}).  One needs about $80$ electrodes in the case of \eqref{eq:elec1} and about $33$ electrodes in the case of \eqref{eq:elec2} to reach noise-to-signal ratios that are less than one percent for all four considered current patterns and our particular choice of the conductivity profile. As it is expected that the constant appearing in \eqref{eq:hypothesis} is larger when $R$ is close to one, the chosen example illustrates a situation that is presumably difficult to resolve using electrode models.

Observe that it is easy to simulate $\Upsilon(\varsigma) f_n$ for any $n \in \Z\setminus \{0\}$ because $\{ f_n \}_{n \in \Z \setminus \{0\} }$ are precisely the eigenfunctions of $\Upsilon(\varsigma)$ with the corresponding eigenvalues
\begin{equation}
  \label{eq:eigenvalues}
\lambda_n = \frac{2}{|n|} \frac{ \tfrac{1-\kappa}{1+\kappa}R^{2|n|}}{1-\tfrac{1-\kappa}{1+\kappa}R^{2|n|}}, \qquad n \in \Z\setminus \{0\}.
\end{equation}
The approximating CEM-based outputs $\hat{\Upsilon}_{M}(\varsigma) f_n$, $n=1,2, 4, 8$, can be numerically computed by resorting to a variant of the Fourier-based numerical forward solver used in \cite{Hyvonen09}; see \cite{Somersalo1992} for the original ideas behind this approach. To be slightly more precise, the absolute CEM measurements corresponding to the homogeneous unit disk can be approximated exactly as in~\cite{Hyvonen09}, while those corresponding to the embedded concentric inhomogeneity require the modifications listed in \cite[Remark~6.1]{Hyvonen2009b} to the solver of \cite{Hyvonen09}, but with $\rho^{2|j|}$ replaced by $\tfrac{1-\kappa}{1+\kappa} R^{2|j|}$ at every occurrence in \cite[Remark~6.1]{Hyvonen2009b}.
\end{example}

\begin{example}
  \label{ex:second}
  We continue to assume the examined domain is the unit disk, but this time around the conductivity inside $D$ is characterized by a nonconcentric discoidal inclusion:
  \begin{equation}
    \label{eq:cond_noncon}
  \varsigma :=
\left\{
\begin{array}{ll}
  1 & \text{in } D \setminus \overline{D_{c,\rho}}, \\[1mm]
  \kappa & \text{in } D_{c,\rho}
\end{array}
\right.
\end{equation}
where $\kappa \in \R_+$ and $D_{c,\rho}$ is the open disk of radius $0 < \rho < 1 - |c|$ centered at $c \in D$. We choose $\kappa = 0.5$,  $c = -0.4+0.2\,{\rm i}$ and $\rho = 0.4$ in \eqref{eq:cond_noncon}, assume point electrodes are attached to $\partial D$ at the equiangular positions $x_m$, $|m| \leq M$, given by \eqref{eq:pointe_D}, and aim to verify the convergence rate in Corollary~\ref{cor:exponential}. The left-hand image of Figure~\ref{fig:figure2} shows the considered inclusion $D_{c, \rho}$ inside the unit disk as well as $2M+1$ equidistant point electrodes on $\partial D$ for $M=16$.

\begin{figure}[htb]
  \includegraphics[width=70mm]{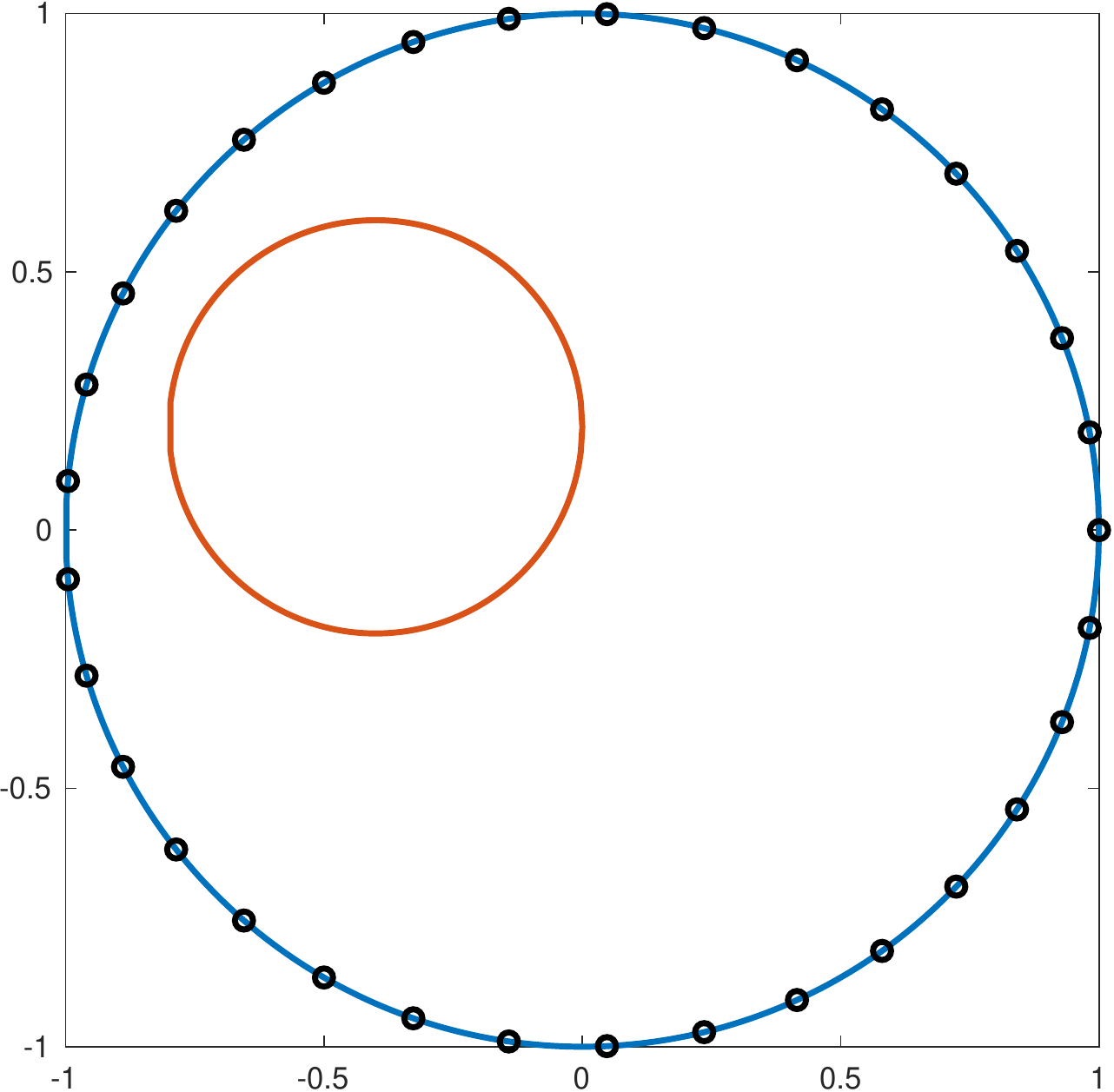} \qquad \includegraphics[width=70mm]{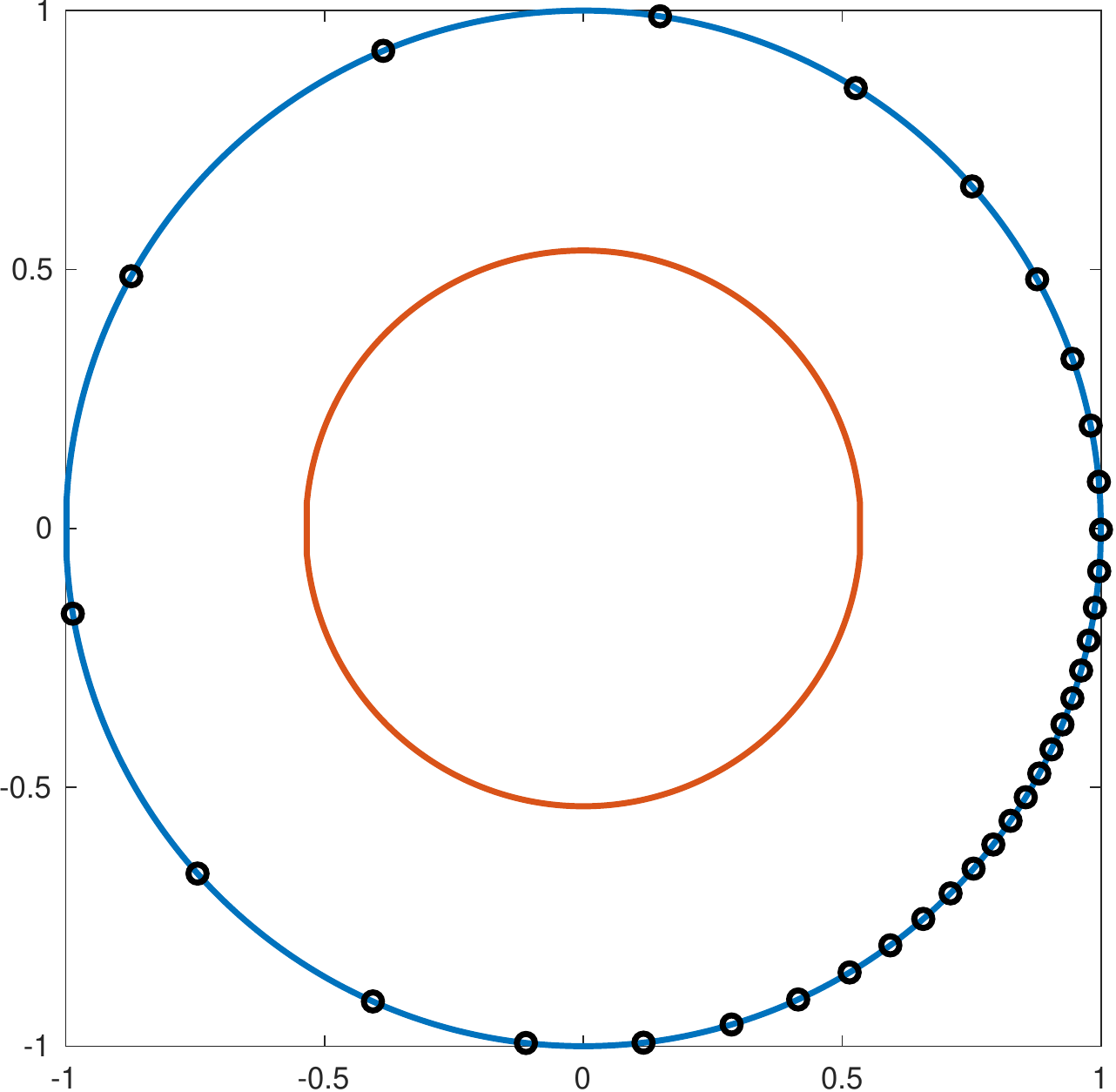}
  \caption{Left: the nonconcentric inclusion $D_{c, \rho}$ inside the unit disk with $2 M + 1$ equidistant point electrodes on $\partial D$ for $c = -0.4+0.2\,{\rm i}$, $\rho = 0.4$ and $M=16$. Right: a corresponding concentric geometry obtained as an image of the nonconcentric one under a suitable M\"obius transformation.} \label{fig:figure2}
\end{figure}

The top left image of Figure~\ref{fig:figure3} shows two mean free current patterns $g_j \in L^2_\diamond(\partial D)$, $j=1, 2$, as functions of the polar angle. The top right and bottom left images present the resulting relative continuum measurements $\Upsilon(\varsigma) g_j$, $j=1,2$, as well as the corresponding PEM-based approximations $\Upsilon_M(\varsigma) g_j$, $j=1,2$, for $M=4$ and $M=8$. The relative (mean free) potentials at the employed point electrodes are depicted along the respective curves obtained through trigonometric interpolation. Notice that the employed approximative boundary operator can be written directly as $\Upsilon_M(\varsigma) = P_\diamond Q_M \Upsilon(\varsigma) F_M G_M$, as in Corollary~\ref{cor:exponential}, since we are dealing with the unit disk and equiangular  point electrodes. For the less regular current density $g_1$, the approximation given by $\Upsilon_4(\varsigma)g_1$ is still quite bad and it exhibits oscillations typical for trigonometric interpolation, while $\Upsilon_8(\varsigma)g_1$ and $\Upsilon(\varsigma) g_1$ are already in a fairly good, yet not perfect agreement. On the other hand, although the approximation $\Upsilon_4(\varsigma)g_2 \approx \Upsilon(\varsigma) g_2$ for the shifted low Fourier mode $g_2$ still leaves a lot to hope for, the 17 point electrodes corresponding to $M=8$ seem to already provide a sufficient approximation since differentiating between $\Upsilon_8(\varsigma)g_2$ and $\Upsilon(\varsigma) g_2$ by a naked eye is almost impossible.

\begin{figure}[htb]
  \includegraphics[width=70mm]{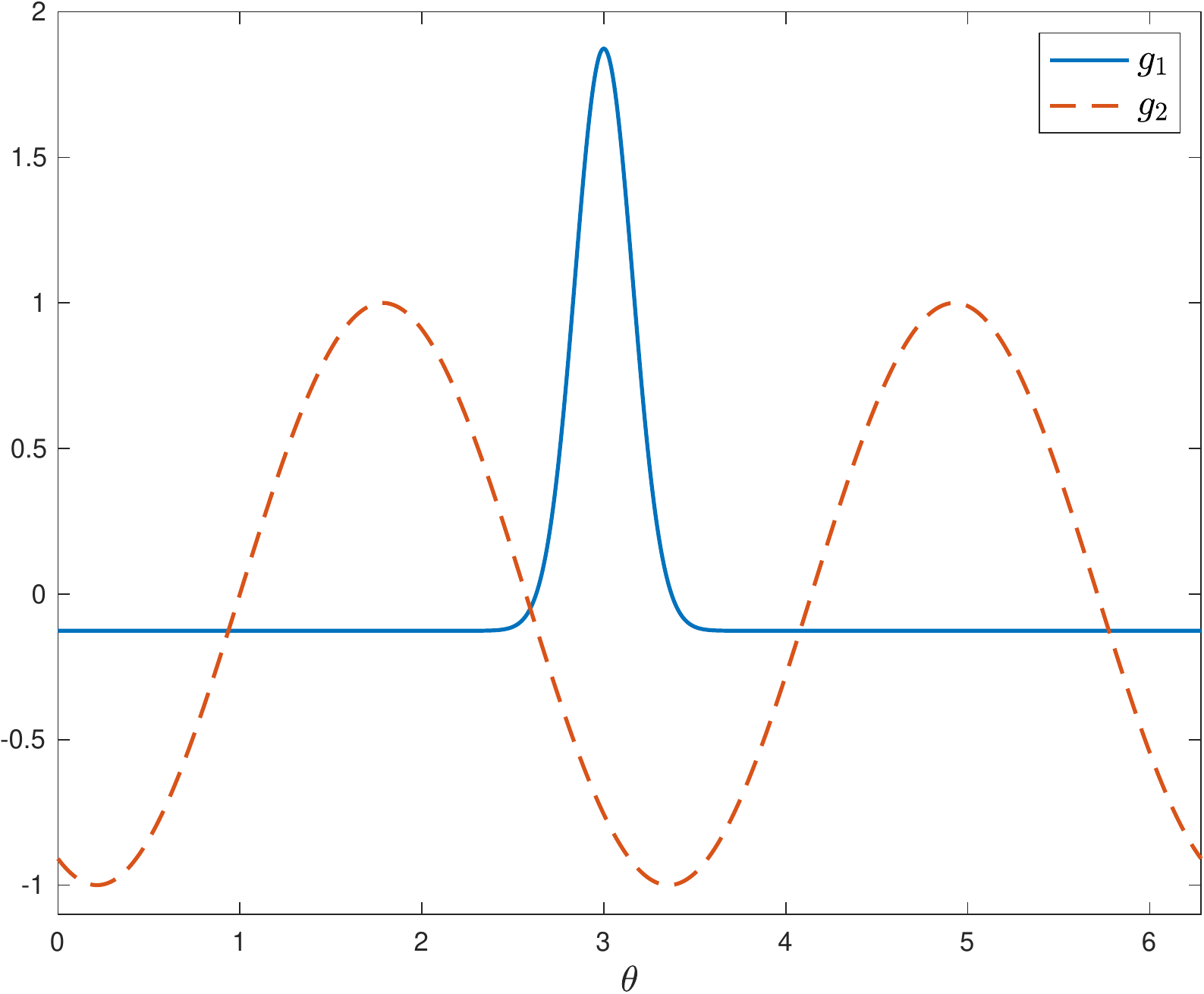} \qquad \includegraphics[width=70mm]{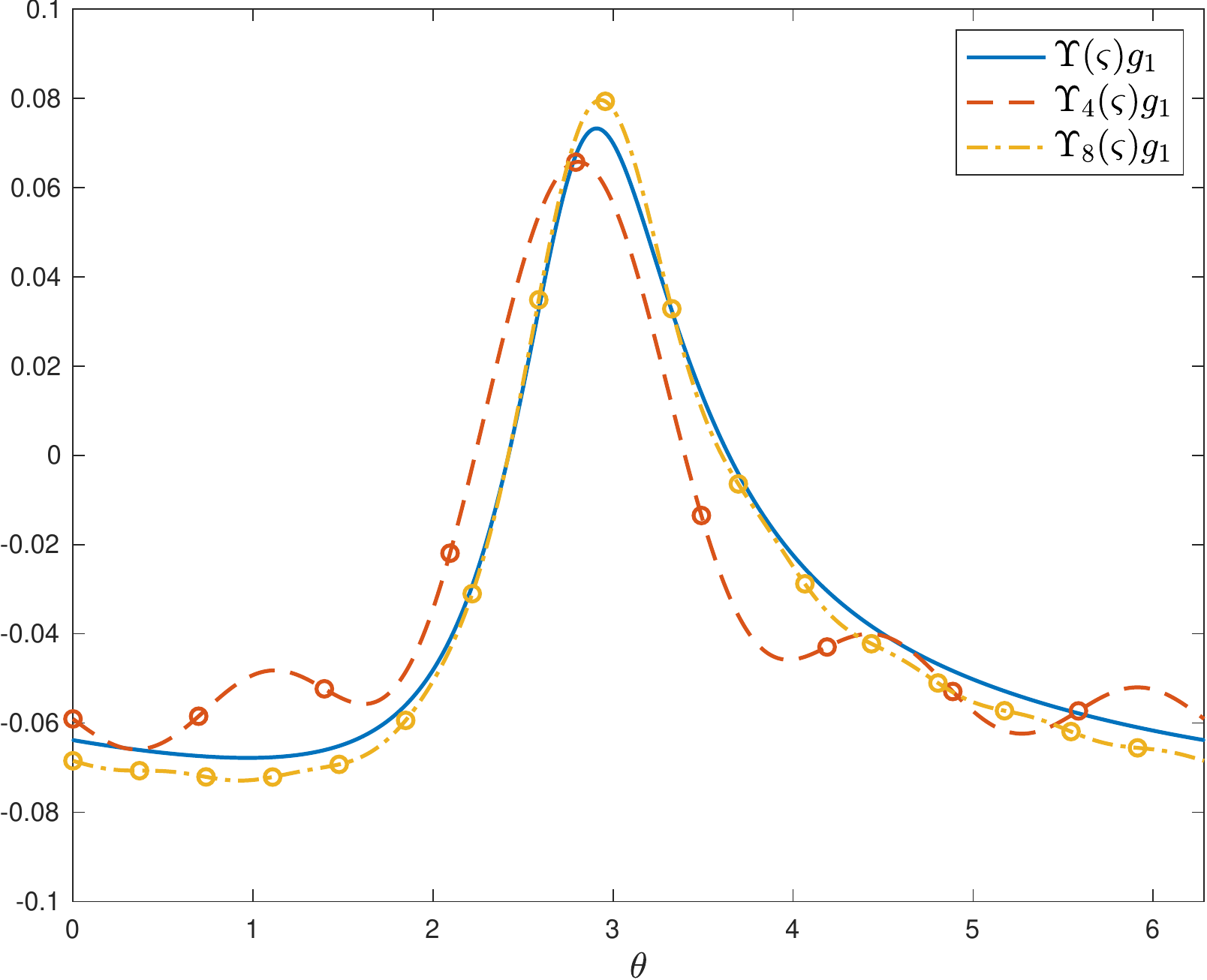} \\[3mm]
  \includegraphics[width=70mm]{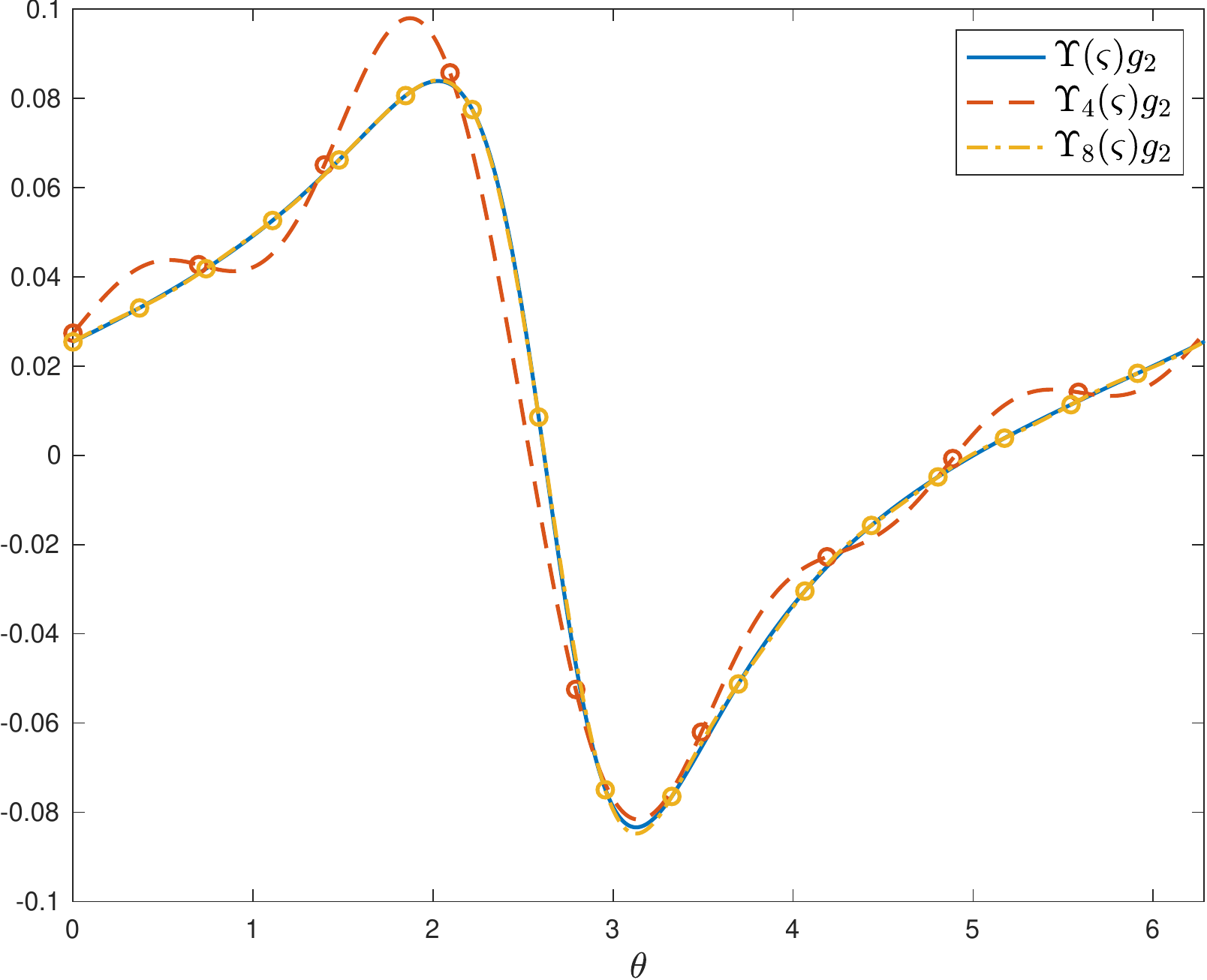} \qquad \includegraphics[width=71mm]{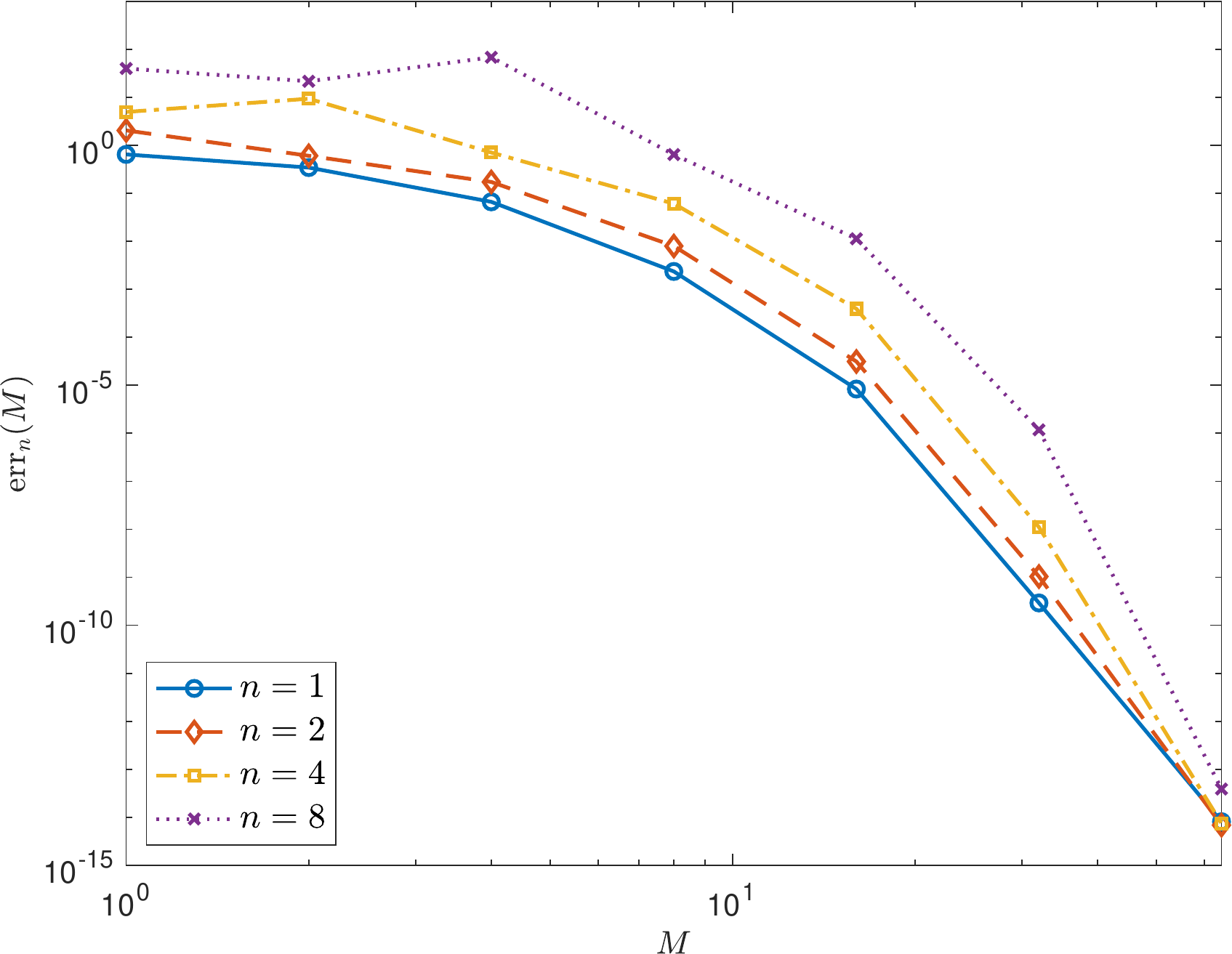}
  \caption{The considered conductivity is $\varsigma$ from \eqref{eq:cond_noncon} with $\kappa = 0.5$, $c = -0.4 + 0.2 \, {\rm i}$ and $\rho =0.4$. Top left: two mean free current densities $g_j$, $j=1,2$, as functions of the polar angle. Top right ($j=1$) and bottom left ($j=2$): the relative continuum measurement $\Upsilon(\varsigma) g_j$ (solid) and the corresponding PEM-based approximation $\Upsilon_M(\varsigma) g_j$ for $M=4$ (dashed) and $M=8$ (dash-dotted). Bottom right: the relative $L^2(\partial D)$ discrepancy ${\rm err}_n(M)$ from \eqref{eq:relative_error}, with $\hat{\Upsilon}_M(\varsigma)$ replaced by $\Upsilon_M(\varsigma)$, as a function of $M$. Circles and solid blue line: $n=1$. Diamonds and dashed red line: $n=2$. Boxes and dash-dotted yellow line: $n=4$. Crosses and dotted magenta line: $n =8$.} \label{fig:figure3}
\end{figure}

Let us then numerically investigate the actual convergence rate in Corollary~\ref{cor:exponential}. To this end, we redefine the relative $L^2(\partial D)$ error ${\rm err}_n(M)$, originally introduced in \eqref{eq:relative_error}, by replacing $\hat{\Upsilon}_M(\varsigma)$ with $\Upsilon_M(\varsigma)$ and using the conductivity $\varsigma$ from \eqref{eq:cond_noncon}, as we consider approximations based on the PEM, not the CEM, and a nonconcentric geometry in this example. The bottom right image of Figure~\ref{fig:figure3} shows ${\rm err}_n(M)$ as a function of $M$ for $n=1, 2, 4, 8$. The concave shape of the convergence plots on a log-log scale is in line with the prediction of Corollary~\ref{cor:exponential} that the convergence rate in $L^2(\partial \Omega)$ is faster than any negative power of $M$ for a fixed Fourier frequency. Note that the maximum number of electrodes in the bottom right image of Figure~\ref{fig:figure3} is 129 corresponding to $M=64$, which is only about one quarter of the number of electrodes used in Figure~\ref{fig:figure1} for the finite-sized electrode of CEM. This is due to the extremely fast convergence of the relative $L^2(\partial D)$ error ${\rm err}_n(M)$ for the PEM: at $M=64$ the relative error is already less than $10^{-13}$ for all four considered Fourier frequencies.

We conclude this example by briefly explaining how (relative) CM and PEM measurements were simulated with high accuracy in this setting by combining the Riemann mapping theorem for doubly-connected domains with the eigendecomposition of the relative continuum ND map characterized by \eqref{eq:eigenvalues} in a concentric geometry: There exist a conformal mapping, or more precisely a M\"obius transformation $\Xi: D \to D$ that sends $D$ onto itself and $D_{c,\rho}$ onto a concentric open disk $D_{0,R}$, with $0 < R = R(|c|,\rho)<1$ uniquely determined by $|c|$ and $\rho$; see,~e.g.,~\cite{Schinzinger2003}. This M\"obius transformation extends conformally to an open neighborhood of $\overline{D}$ and it is unique up to rotations of its image. Let us denote the inverse of $\Xi$ by $\Theta$. If $u\in H^1(D)$ satisfies the conductivity equation for the conductivity \eqref{eq:cond_noncon} with the Neumann trace $f \in L^2_\diamond(\partial D)$, then $\tilde{u} := u \circ \Theta \in H^1(D)$ also satisfies the conductivity equation but for the current density $|\Theta'| (f \circ \Theta)|_{\partial D} \in L^2_\diamond(\partial D)$ and the radially symmetric conductivity \eqref{eq:cond_con} with the above introduced $R = R(|c|,\rho)$; see \eqref{eq:PDE_D2} and \cite[Lemma~4.1 and Remark~4.1]{Hyvonen2018}. If $D$ is characterized by the unit conductivity instead of \eqref{eq:cond_noncon}, these conclusions remain valid with the exception that the conductivity does not alter when mapped conformally. As a consequence, one can numerically simulate the relative CM measurement for a given current density $f \in L^2_\diamond(\partial D)$ and the conductivity \eqref{eq:cond_noncon} by first presenting $|\Theta'| (f \circ \Theta)|_{\partial D}$ as a Fourier series, then utilizing \eqref{eq:eigenvalues} to obtain the corresponding relative measurement in the concentric geometry, and finally mapping the relative potential back to the original nonconcentric geometry by composing it with $\Xi|_{\partial D}$.

The simulation of relative PEM data can be performed following the same basic idea, but bearing in mind that Neumann conditions involving Dirac delta distributions transfer naturally under conformal mappings. That is, one need not involve the multiplier $|\Theta'|$ in the boundary current density when transferring it to the concentric geometry but only to move the point electrodes as dictated by the boundary restriction of the conformal mapping; see~\eqref{eq:PDE_O}, \eqref{eq:PDE_D3} and \cite[Theorem~3.2]{Hakula2011}. Because the Fourier coefficients for a linear combination of Dirac delta distributions are obtained via \eqref{eq:F_delta}, one can readily apply \eqref{eq:eigenvalues} to obtain the corresponding relative potential for the concentric conductivity \eqref{eq:cond_con} on $\partial D$. Evaluating this potential at the images of the original point electrodes under $\Xi|_{\partial D}$ finally produces the sought-for PEM data. The claimed accuracy of such a numerical approach to simulating relative PEM data is explained by the geometric convergence of the eigenvalues in \eqref{eq:eigenvalues}. The right-hand image of Figure~\ref{fig:figure2} shows a concentric geometry corresponding to the nonconcentric one in the left-hand image, obtained with the help of a certain M\"obius transformation $\Xi$ carrying the above listed properties.

\end{example}

\section{Concluding remarks}

In three dimensions, one cannot straightforwardly first introduce the above theory for a reference domain, such as the unit disk in two dimensions, and then transfer the obtained estimates to more general domains with the help of suitable diffeomorphisms. However, given a smooth and bounded three-dimensional domain and an interpolation strategy on its boundary, the above line of reasoning can be repeated as such to obtain estimates of the same order as is the accuracy of the employed interpolation operator (and the related quadrature rule). Choosing optimal electrode positions in three dimensions thus boils down to introducing accurate quadrature and interpolation rules for a given two-dimensional surface. As an example, one could use interpolation by spherical harmonics with respect to some suitable set of point electrodes on the unit sphere; see,~e.g.,~\cite{Delbary2014}. We leave further analysis of the three-dimensional setting for future studies.

Observe that the results on the discrepancy between $\Upsilon(\sigma)$ and $\Upsilon_M(\sigma)$ or $\hat{\Upsilon}_M(\sigma)$ presented in this work can be interpreted as estimates on the (numerical approximation) noise introduced when relative CM measurements of EIT are approximated by practical electrode measurements. Hence, replacing $\Upsilon(\sigma)$ by $\Upsilon_M(\sigma)$ or $\hat{\Upsilon}_M(\sigma)$ in a reconstruction algorithm designed for relative CM measurements of EIT can be expected to have an approximately similar effect on the reconstruction quality as adding an equivalent amount of artificial measurement noise to simulated relative CM measurements. One can thus combine the above results with analysis on the noise sensitivity of reconstruction algorithms designed for relative CM measurements presented in previous works to approximately deduce how many (optimally positioned) electrodes need to be introduce on the object boundary to make the approximation noise level low enough to allow such algorithms to function satisfactorily.

\subsection*{Acknowledgments} 

This work was supported by the Academy of Finland (decision 312124) and the Aalto Science Institute (AScI).

\appendix

\section{On the discrepancy between the CEM and the PEM}

The purpose of this appendix is to show that the constant appearing on the right-hand side of the estimate in \cite[Corollary~2.1]{Hanke2011b} depends on the number of electrodes as $\mathcal{O}(M)$ when $M$ tends to infinity. This result is used for the proof of Corollary~\ref{cor:main}. We achieve this by carefully analyzing the relevant intermediate steps in \cite{Hanke2011b}.
The following theorem is simply a two-dimensional version of \cite[Corollary~2.1]{Hanke2011b} with our notation and the dependence on $M$ given explicitly; to be quite precise, in  \cite{Hanke2011b} the number of electrodes is $M$, not $2M+1$ as in our case, but this obviously has no effect on the asymptotic form of the actual estimate. 

\begin{theorem}
  \label{thm:M-point}
  Let the assumptions of Section~\ref{sec:general_domain} hold. In particular, the finite-sized electrodes of CEM are assumed to satisfy \eqref{eq:el_length} with $d_0 = d_0(M,\Omega)>0$ that guarantees the electrodes do not overlap.  Then,
  \begin{equation*}
  	\norm{\Upsilon_{\rm CEM} (\sigma) - \Upsilon_{\rm PEM} (\sigma)}_{\mathscr{L}(\C_\diamond^{2M+1})} \leq C M d^2,	
  \end{equation*}
  where $\C^{2M+1}$ is equipped with the Euclidean norm and $C = C(\Omega, \sigma, z)> 0$ is independent of $M \in \N$ and $d \in (0, d_0)$.
\end{theorem}

\begin{proof}
  Let us re-emphasize that the assertion is precisely the two-dimensional version of \cite[Corollary~2.1]{Hanke2011b} using our current notation and with the dependence of the right-hand side on the number of electrodes written out explicitly. In the rest of this proof, we will follow the notation of \cite{Hanke2011b} to which we also refer for more information; in particular, the number of electrodes is $M \in \N \setminus \{1\}$, the parameter controlling the width of the electrodes is $h$ instead of $d$, and $n=2$ or $3$ is the spatial dimension. The idea is to simply go through the relevant estimates in \cite{Hanke2011b} and deduce the dependence of the appearing constants on $M$.

  Let us start with \cite[Lemma~2.1]{Hanke2011b}. It is easy to check that the right-most constant in \cite[(2.9)]{Hanke2011b} depends on $M$ as $\mathcal{O}(M^{1/2})$ due to the relationship between 1 and 2-norms in the $M$-dimensional vector space $\C^M$; observe that the vector norm $\| \cdot \|_{\C^M}$ used for $\C^M$ in \cite{Hanke2011b} is Euclidean, although this choice plays no role in the analysis of \cite{Hanke2011b} since the limit $M\to \infty$ is not considered there, and thus all vector norms are `uniformly equivalent' in \cite{Hanke2011b}. The aforementioned  dependence on $M$ carries over to the estimate of \cite[Lemma~2.1]{Hanke2011b}, which now reads
  \begin{equation}\label{eq:appendix0}
  	\norm{w}_{H^r(\partial \Omega)/ \C} \leq C M^{1/2} \norm{I}_{\C^M},
  \end{equation}
  with the constant $C = C(\Omega, \sigma, r) > 0$ being independent of $M$.

Because the constant  in \cite[Lemma~3.1]{Hanke2011b} is completely independent of the electrode configuration, the next result we need to check is \cite[Lemma~3.2]{Hanke2011b}. We start by analyzing the two terms on the right-hand side of \cite[(3.2)]{Hanke2011b}, starting with the first one:
\begin{align*}
  \sum_{m=1}^M  \Big\| f^h - \frac{I_m}{|e^h_m|} \Big\|_{L^2(e^h_m)}  \| \varphi - \varphi(x_m) \|_{L^2(e^h_m)} & \leq C h^{(n+3)/2} \| \varphi \|_{C^1(\partial \Omega)} \sum_{m=1}^M \| f^h \|_{H^1(e^h_m)} \\
  &\leq C M^{1/2} h^{(n+3)/2}   \| \varphi \|_{C^1(\partial \Omega)} \| f^h \|_{H^1(\cup e^h_m)},
\end{align*}
where we first used \cite[(3.4)]{Hanke2011b} together with the first inequality in \cite[(3.3)]{Hanke2011b}, and then introduced $M^{1/2}$ to bound a certain $1$-norm by the corresponding $2$-norm. In consequence,
\begin{align}
  \label{eq:appendix1}
  \sum_{m=1}^M \Big\| f^h - \frac{I_m}{|e^h_m|} \Big\|_{L^2(e^h_m)} \| \varphi - \varphi(x_m) \|_{L^2(e^h_m)} \leq C M^{1/2}  h^{2} \| \varphi \|_{C^1(\partial \Omega)} \| I \|_{\C^M}
\end{align}
due to \cite[Lemma~3.1]{Hanke2011b} and the assumed dependence of $|e^h_m|$ on the parameter $h$ controlling the size of the electrodes (cf.~ \cite[(2.15)]{Hanke2011b} and \eqref{eq:el_length}). On the other hand, the second term on the right-hand side of \cite[(3.2)]{Hanke2011b} can be estimated as
\begin{equation}
   \label{eq:appendix2}
 \sum_{m=1}^M \frac{|I_M|}{|e^h_m|} \, \Big| \int_{e^h_m}( \varphi - \varphi(x_m)) \, {\rm d} S \Big| \leq C h^{n+1} \| \varphi \|_{C^2(\partial \Omega)} \sum_{m=1}^M \frac{|I_m|}{|e^h_m|} \leq C M^{1/2} h^2  \| \varphi \|_{C^2(\partial \Omega)} \|I \|_{\C^M},
\end{equation}
where the first step directly follows from \cite[(3.6)]{Hanke2011b}  and the second one is a consequence of the assumed dependence of $|e^h_m|$ on $h$ as well as the relationship between $1$ and $2$-norms. By plugging \eqref{eq:appendix1} and \eqref{eq:appendix2} in \cite[(3.2)]{Hanke2011b} and utilizing the resulting estimate in part (5) of the proof of \cite[Lemma~3.2]{Hanke2011b}, one finally sees that the $M$-dependent version of the estimate in \cite[Lemma~3.2]{Hanke2011b} reads
\begin{equation}
  \label{eq:appendix3}
\| \nu \cdot \sigma \nabla (u^h - u) \|_{H^{-(n+3)/2 - \varepsilon}(\partial \Omega)} \leq C_{\varepsilon} M^{1/2} h^2 \|I \|_{\C^M},
\end{equation}
where the assumptions and definitions are as in \cite[Lemma~3.2]{Hanke2011b}.

Next we need to consider  \cite[Corollary~3.1]{Hanke2011b}. The $M$-dependent estimate of \cite[Lemma~3.2]{Hanke2011b}, i.e.~\eqref{eq:appendix3}, immediately leads to the constant on the right-hand side of \cite[(3.7)]{Hanke2011b} depending on $M$ as $\mathcal{O}(M^{1/2})$. This carries directly over to the first estimate of \cite[Corollary~3.1]{Hanke2011b} that now reads
\begin{equation}
  \label{eq:appendix4}
\| u^h - u \|_{H^1(\Omega_0)/\C} \leq C M^{1/2} h^2 \| I \|_{\C^M}.
\end{equation}
Similarly, the second estimate of \cite[Corollary~3.1]{Hanke2011b} takes the form
\begin{equation}
  \label{eq:appendix5}
\| u^h \|_{H^1(\Omega_0)/\C} + \| u \|_{H^1(\Omega_0)/\C} \leq C M^{1/2} \| I \|_{\C^M},
\end{equation}
when the corresponding argument in the proof of \cite[Corollary~3.1]{Hanke2011b} is repeated with the $M$-dependent version of \cite[(2.9)]{Hanke2011b} that includes the term $M^{1/2}$ on its right-hand side as mentioned above.

The proof for the $M$-dependent version of \cite[Corollary~2.1]{Hanke2011b} can now be completed by using \eqref{eq:appendix4} and \eqref{eq:appendix5} for the final estimate in the proof of \cite[Theorem~2.1]{Hanke2011b}, leading to the dependence $\mathcal{O}(M^{1/2}) \cdot \mathcal{O}(M^{1/2}) = \mathcal{O}(M)$ on the right-hand side. To be quite precise, this proves an $M$-dependent version of \cite[Theorem~2.1]{Hanke2011b}, that is,
$$
\| (U^h - U^h_0) - W \|_{\C^M/\C} \leq C M h^2 \| I \|_{\C^M},
$$
where $C = C(\Omega, \sigma, z)>0$ is independent of $M$ and $h \in (0, h_0)$ for some $h_0 > 0$. However, \cite[Corollary~2.1]{Hanke2011b} is just a restatement of \cite[Theorem~2.1]{Hanke2011b} using the appropriate operator norm. Finally, translating the $M$-dependent \cite[Corollary~2.1]{Hanke2011b} into the notation of this paper completes the proof.
\end{proof}

\bibliographystyle{plain}
\bibliography{minbib}

\end{document}